\documentclass[onefignum,onetabnum]{siamonline190516}

\usepackage[utf8]{inputenc}

\usepackage{amsmath}
\usepackage{amssymb}
\usepackage{todonotes}
\usepackage{subcaption}
\usepackage{wrapfig}
\usetikzlibrary{patterns}
\usepackage{pgfplots,xcolor}
\pgfplotsset{compat=1.15}
\usepackage{enumitem,linegoal}

\DeclareGraphicsExtensions{.pdf,.jpeg,.png}

\newcommand{\off}[1]{}

\newtheorem{thm}{Theorem}[section]
\newtheorem{rem}[thm]{Remark}

\newtheorem{example}[thm]{Example}
\newtheorem{assumption}{Assumption}
\newtheorem{observation}{Observation}

\numberwithin{equation}{section}

\newcommand{\Z}{\mathbb{Z}}
\newcommand{\R}{\mathbb{R}}
\newcommand{\N}{\mathbb{N}}
\renewcommand{\H}{\mathcal{H}}

\newcommand{\norm}[1]{\left\Vert #1 \right\Vert}
\newcommand{\abs}[1]{\left\vert #1 \right\vert}

\newcommand{\argmin}{\mathrm{arg}\min}

\newcommand{\st}{\,:\,}

\newcommand{\dx}{\,\mathrm{d}x}
\renewcommand{\d}{\,\mathrm{d}}

\newcommand{\calN}{\mathcal{N}}

\newcommand{\eps}{\varepsilon}
\newcommand{\id}{\mathrm{id}}
\newcommand{\tv}{\mathrm{TV}}

\newcommand{\preim}{\mathrm{preim}}
\newcommand{\ran}{\mathrm{ran}}
\newcommand{\prox}{\operatorname{prox}}
\newcommand{\proj}{\operatorname{proj}}
\newcommand{\relu}{\mathrm{ReLU}}
\renewcommand{\div}{\operatorname{div}}
\definecolor{darkgreen}{rgb}{0.,.7,0.}

\definecolor{plum}{rgb}{0.56, 0.27, 0.52}
\graphicspath{
  {figs/}
}

\newcommand{\revise}[1]{\textcolor{red}{#1}}
\renewcommand{\revise}[1]{#1}

\newcommand{\revision}[1]{{\color{red}#1}}
\renewcommand{\revision}[1]{#1}

\renewcommand{\thefootnote}{\fnsymbol{footnote}}

\begin{document}

\title{Nonlinear Power Method for Computing Eigenvectors \\of Proximal Operators and Neural Networks}

\author{Leon Bungert\footnotemark[1]\  \footnotemark[4]
\and Ester Hait-Fraenkel\footnotemark[2]\ \footnotemark[4] 
\and Nicolas Papadakis\footnotemark[3]
\and Guy Gilboa\footnotemark[2]}



\renewcommand{\thefootnote}{\fnsymbol{footnote}}
\footnotetext[1]{Department of Mathematics, University of Erlangen (\email{leon.bungert@fau.de})}
\footnotetext[2]{Department of Electrical Engineering, Technion - Israel Institute of Technology}
\footnotetext[3]{Univ. Bordeaux, IMB, Bordeaux INP, CNRS, UMR 5251 F-33400 TALENCE, France}
\footnotetext[4]{These authors contributed equally to this work.}
\maketitle


\begin{abstract}
    Neural networks have revolutionized the field of data science, yielding remarkable solutions in a data-driven manner. 
    For instance, in the field of mathematical imaging, they have surpassed traditional methods based on convex regularization.
    However, a fundamental theory
    supporting the practical applications is still in the early stages of development. 
    We take a fresh look at neural networks and examine them via nonlinear eigenvalue analysis.
    The field of nonlinear spectral theory is still emerging, providing insights about nonlinear operators and systems. 
    In this paper we view a neural network as a complex nonlinear operator and attempt to find its nonlinear eigenvectors. 
    We first discuss the existence of such eigenvectors and analyze the kernel of $\relu$ networks.
    Then we study a nonlinear power method for generic nonlinear operators. 
    For proximal operators associated to absolutely one-homogeneous convex regularization functionals, we can prove convergence of the method to an eigenvector of the proximal operator.
    This motivates us to apply a nonlinear method to networks which are trained to act similarly as a proximal operator.
    In order to take the non-homogeneity of neural networks into account we define a modified version of the power method.
    \\
    We perform extensive experiments \revise{for different proximal operators and} on various shallow and deep neural networks designed for image denoising.
    \revise{Proximal eigenvectors will be used for geometric analysis of graphs, as clustering or the computation of distance functions.}
    For simple neural nets, we observe the influence of training data on the eigenvectors.
    For state-of-the-art denoising networks, we show that eigenvectors can be interpreted as (un)stable modes of the network, when contaminated with noise or other degradations.
\end{abstract}

\section{Introduction}
\label{Sec::intro}
The emerging field of nonlinear spectral theory allows better understanding of nonlinear processes, as well as designing algorithms based on nonlinear spectral methods.
In this paper, we make a first step towards understanding neural networks via their (nonlinear) eigenvectors. 
In image processing and learning, there were several theoretical advances in analyzing nonlinear eigenproblems. \cite{Meyer[1],TVFlowInRN} formulated analytic solutions for eigenfunctions associated to the $1$-Laplacian. \cite{benning2013ground} analyzed properties of ground states of one-homogeneous functionals. A nonlinear spectral decomposition based on total-variation proposed in \cite{gilboa2013spectral,Gilboa_spectv_SIAM_2014} was later generalized to the one-homogeneous case in \cite{burger2016spectral}, with theory for the discrete case. 
Recently, \cite{bungert2019nonlinear,bungert2019computing} rigorously analyzed this framework in the infinite-dimensional setting \revise{and showed that decompositions into nonlinear eigenvectors can only be obtained under special geometrical conditions}. 
A $p$-Laplacian spectral framework is formulated in \cite{cohen2020introducing}. Applications to image denoising \cite{MoellerICCV2015}, segmentation \cite{ZeuneSegmentation2017}, fusion \cite{hait2019spectral} and classification \cite{aviles2019} were proposed.

A very difficult problem for general nonlinear operators is how to compute their eigenvectors.
In the context of learning, the authors of \cite{BressonSzlam2010Cheeger} estimated the Cheeger cut on graphs by a Rayleigh-type quotient. This was later generalized in \cite{hein2010inverse} to a nonlinear inverse-power method.
\cite{nossek2018flows,aujol2018theoretical,bungert2019asymptotic} suggested nonlinear flows to solve eigenproblems induced by total-variation and general one-homogeneous functionals. 
This was later generalized to solve eigenproblems emerging in nonlinear optics \cite{cohen2018energy}.
Algorithms to minimize generalized Rayleigh-quotients on grids and graphs were proposed and analyzed in \cite{feld2019rayleigh}. 
Furthermore, in \cite{effland2019optimal,kobler2020total} the authors computed nonlinear eigenfunctions associated to learned convolutional regularizing functionals, which generalize total variation.
\revise{%
Nonlinear power methods and Perron-Frobenius theory for order-preserving multihomogeneous maps were analyzed in \cite{gautier2019perron,gautier2019unifying,gautier2020computing} to generalize the classical results for the computation of matrix norms \cite{boyd1974power}.}
The methods above assume either that the operator is a subgradient of a convex functional, or at least an analytically known operator.
Significantly more complex nonlinear operators were only recently studied for the first time, for the case of nonlinear denoisers, as suggested by part of the authors in \cite{HAITFRAENKEL2021103041}. 

Neural networks have revolutionized the world of computer vision and image processing \cite{egmont2002image}, applied for many tasks, such as classification and segmentation (cf.~\cite{schmidhuber2015deep} for an overview), depth estimation \cite{liu2015deep,garg2016unsupervised}, tracking \cite{fan2010human,nam2016learning}, to name a few. 
An ongoing extensive research on mathematical frameworks aims to interpret neural nets. 
This includes earlier studies in the context of wavelets and a generalization of the scatter transform \cite{mallat2016}, the interpretation of residual networks as nonlinear ODEs \cite{haber2017stable,chen2018neural}, deep layer limits \cite{thorpe2018deep}, sparse coding \cite{papyan2018theoretical} and more. 
Recent studies, more closely related to our work, are the SVD analysis of a ReLU layer \cite{dittmer2019singular} and convergence of plug-and-play ADMM algorithms using denoising networks with certain Lipschitz regularity \cite{ryu2019plug}.

The aim of this paper is twofold: firstly, we rigorously analyze a nonlinear power method of proximal operators, whose eigenvalue problem is in fact equivalent to the eigenvalue problems associated to Rayleigh quotients.
Secondly, we generalize this power method to compute eigenvectors of general neural networks, which do not possess the homogeneity properties of proximal operators. Similarly to many of the works above, we here focus on networks with inputs and outputs of the same dimension. 
That is, given an input image, the output of the network is an image of the same size, which is common in many image-processing tasks.
More specifically, although our algorithms are more general, we direct our efforts towards denoising networks (e.g., \cite{zhang2017beyond,zhang2018ffdnet}).
Given a noisy image, such nets estimate a suitable clean image. 
In this setting a nonlinear eigenvalue problem can be well defined, along with some general regularity assumptions on the behavior of the net.

\revision{%
While proximal operators and denoising neural networks obviously are mathematically very different objects, they both represent modern approaches for denoising images.
Proximal operators are model-based and are designed in order to minimize a certain energy functional.
In contrast, neural networks are parametric nonlinear mappings, where the parameters are optimized using huge amounts of data.
Hence, their behavior strongly depends on the kind of training data used.
An important similarity between proximal operators and neural networks for image denoising is that they map their input space into itself and reduce the complexity of their inputs. 
These two aspects would not be satisfied, for instance, for classification networks, which map into a small label space, or deblurring networks, which sharpen images and can be expected to increase their complexity.
Furthermore, as shown in \cite{meinhardt2017learning}, denoising neural networks can replace proximal operators in optimization algorithms for solving inverse problems. Recent works \cite{xu2020provable} also propose to design denoising networks that approximate the minimum mean squared error denoiser, that is proven to be the proximal operator of some (possibly nonconvex) function \cite{gribonval2011should}.  
These strong relations motivate us to analyze power iterations of proximal operators thoroughly and generalize them to neural networks, where a rigorous analysis is currently out of reach because of their highly complex nonlinearity.
Our computational results show that the proposed modification of the power method for neural networks behaves very similarly to linear or proximal power iterations.
}

\revise{Our motivation for analyzing eigenvectors of black-box operators or neural nets is the following: first, eigenvalue analysis provides us with eigenvectors corresponding to large and small eigenvalues. These eigenvectors constitute the most and least suitable inputs for the net, respectively~\cite{HAITFRAENKEL2021103041}. In the case of denoisers, on which we focus here, those eigenvectors can be denoised best or worst. This serves as a tool to learn about the nature and behavior of the underlying operator and reveal otherwise hidden or hardly visible behavior.
Second, eigenvalue analysis can serve as a tool for spectral decomposition (see for instance the work mentioned earlier regarding spectral decomposition, as well as~\cite{katzir,gilboa2018nonlinear}). In the latter, it was discussed how to generate a scale space of a coarsening operator by repeatedly applying it. This scale space is represented through layers, decaying at different rates under application of the operator. Finding eigenfunctions can be the first stage to solving and understanding this representation problem.}

Our main contributions are:
\begin{enumerate}
    \item We propose a generalization of the linear power method, a classical iterative method to compute eigenvectors of matrices.
    \item We provide a rigorous analysis of the method for a certain class of nonlinear proximal operators, showing its validity. 
    \item We discuss existence of eigenvectors of neural nets and characterize the kernel of ReLU nets as convex polyhedra. 
    \item By computing eigenvectors of state-of-the-art denoising networks we gain insights on the most and least stable structures of the net.
\end{enumerate}

\section{Setting and Definitions}\label{Sec::def}
Let $T:\H\to\H$ be a generic (nonlinear) operator on a real Hilbert space $\H$ with norm $\norm{\cdot}$.
In the case of a neural network one typically has $\H=\R^n$, equipped with the Euclidean norm.
We aim at solving the \emph{nonlinear} eigenproblem   
\begin{equation}
\label{eq:EV}
    T(u)=\lambda u,
\end{equation}
where $u\in\H$ and $\lambda\in\R$ denote the eigenvector and eigenvalue, respectively.

\paragraph{Rayleigh quotient} A common notion in linear eigenvalue analysis is the Rayleigh quotient~\cite{horn2012matrix}, defined for symmetric matrices $L$ as:
\begin{equation}
    \label{eq:Rayleigh_lin}
    R_{\mathrm{lin}}(u):=\frac{u^T L u}{u^T  u}= \frac{\langle u,Lu\rangle}{\|u\|^2}.
\end{equation}
The Euler-Lagrange equation of \cref{eq:Rayleigh_lin} results in the eigenproblem. Therefore, an eigenvector is a critical point of the Rayleigh quotient. We can also understand the Rayleigh quotient as a generalized or approximated eigenvalue for any $u$ (not just eigenvectors). For eigenvectors (admitting $Lu=\lambda u$), we get exactly the corresponding eigenvalue
$R_\mathrm{lin}(u)=\lambda$.
In the nonlinear setting one can define a \emph{generalized Rayleigh quotient,}
\begin{equation}
\label{eq:Rayleigh}
    R(u)= \frac{\langle u,T(u)\rangle}{\|u\|^2}.
\end{equation}
When $T(u)$ is a subgradient of an absolutely one-homogeneous functional $J$, meaning $T(u)\in \partial J(u)$, we have $J(u)=\langle u, T(u)\rangle$. 
In this case \cref{eq:Rayleigh} becomes (see~\cite{nossek2018flows}) $R(u)= {J(u)}/{\norm{u}^2}$ and the eigenvalue problem takes the form $\lambda u \in \partial J(u)$.
In \cite{hein2010inverse} similar quotients were used to obtain an inverse-power method. 
In \cite{feld2019rayleigh,bungert2019asymptotic} the minimization of such quotients based on one- and $p$-homogeneous functionals were analyzed and used to solve eigenvalue problems.

\paragraph{Approximate eigenvectors and angle} Another important concept related to numerical solutions of eigenvectors is their approximation. In nonlinear systems one may often reach only an approximation of an eigenvector. We would like to quantify how close a given vector is to a precise eigenvector. One general formulation for any operator $T$, is by the angle (see~\cite{nossek2018flows}).
For eigenvectors, vectors $u$ and $T(u)$ are collinear. Thus their respective angle is either~$0$ (for positive eigenvalues) or~$\pi$ (for negative eigenvalues). Since both $u$ and $T(u)$ are real, eigenvalues are also real. Thus, the angle is a simple scalar measure that quantifies how close $u$ and $T(u)$ are to collinearity. We define the angle $\theta$ between $u$ and $T(u)$ by
\begin{equation}
\label{eq:Theta}
    \cos (\theta) = \frac{\langle u,T(u)\rangle}{\|u\|\| T(u)\|}.
\end{equation}
We discuss denoisers with positive eigenvalues, thus we aim to reach an angle close to 0.

\paragraph{Linear Power Method}
We briefly recall the linear power method algorithm (also known as Von-Mises-iteration~\cite{mises1929praktische}), which we generalize in this work.
Given a matrix $L\in \mathbb{R}^{n \times n}$ which is diagonalizable, the linear eigenproblem is 
\begin{equation}\label{eq:EV_lin}
    Lu=\lambda u.
\end{equation}
A solution to this problem is found by the iterative algorithm given in \cref{alg:lin}. 
Under some conditions the algorithm converges to the eigenvalue $\lambda$ with the greatest absolute value, and to its corresponding eigenvector. 

\begin{algorithm}
\caption{\emph{Linear} Power Method with Matrix $L$}
\label{alg:lin}
~~ \\
\textbf{Input:} $f\in \R^n$, $\varepsilon>0$.
\begin{enumerate}
\item Initialize: $u^0 \gets f/\|f\|$, $\,\,k \gets 1.$
    \item Repeat until $\|u^{k+1}-u^k\|< \varepsilon$:\\ 
    $u^{k+1} \gets \frac{Lu^k}{\|Lu^k\|}$, $\,\,\,k \gets k+1.$
\end{enumerate}
\textbf{Output:} $(u^*, \lambda^*)$, where $u^*=u^k$, $\lambda^*=R(u^*)$, with $R$ defined in \cref{eq:Rayleigh_lin}. 
\end{algorithm}

\paragraph{Proximal Power Method}
\revision{%
Note that applying a power method to differential operators $T(u)=\partial J(u)$ is not particularly relevant since even in the linear case $T(u)=Lu$ they are unbounded in which case power methods diverge.
Therefore, a canonical choice of nonlinear operator is the proximal operator
\begin{align}
    T(u):=\prox_\alpha^J(u):=\argmin_{u\in\H}\frac{1}{2}\norm{v-u}^2+\alpha J(u),
\end{align}
which has the formal expression $T(u)=(\id+\alpha\partial J)^{-1}(u)$. 
On a formal level this already indicates that, even if the operator $\partial J$ is unbounded, the proximal operator is well behaved with eigenvalues between zero and one.
Furthermore, it shows that a proximal power method is an inverse power method of the operator $u\mapsto (\id+\alpha\partial J)(u)$ which has the same eigenvectors as $\partial J$.
In \Cref{Sec::proximal_operators_analysis} we will see that proximal power iterations indeed converge and can be thoroughly analyzed.
}

\section{Networks as Nonlinear Operators - Existence of Eigenvectors and Kernel}
In this section, we first aim at proving existence of eigenvectors for a generic class of nonlinear operators, including most neural nets.
Using Banach's fixed point theorem we will see that Lipschitz continuity suffices to prove existence of eigenvectors.
Secondly, we study a special class of eigenvectors, namely vectors in the kernel of the operator which fulfill \cref{eq:EV} with $\lambda=0$.
These eigenvectors are of special interest since they characterize the most unstable inputs to the neural net.
For example, if the net is supposed to denoise images then the kernel can be interpreted as pure noise images.
While for linear operators and also for subdifferential operators of homogeneous functionals the kernel can be shown to be a linear space \cite{bungert2019asymptotic,bungert2019nonlinear}, this is not true for general nonlinear operators.
In order to develop a first understanding of the kernel of neural nets we characterize the kernel of $\relu$ networks in the single- and multi-layer case.  

\subsection{Existence of eigenvectors}
In order to prove existence of eigenvectors we consider $T$ as an operator $T:U\to U$, where $U\subset\H$ is some closed subset of $\H$ which meets
\begin{align}\label{eq:scaling_U}
    cU\subset U,\quad\forall 0\leq c\leq 1.
\end{align}
Relevant examples for neural nets are the non-negative orthant $U=\R^n_+$, the unit cube $U=[0,1]^n$, \revision{the unit simplex $U=\{u\in\R^n\st \sum_{i=1}^n u_i \leq 1,\,u_i\geq 0\,\forall i\}$}, or the whole space $U=\R^n$. 
The infinite-dimensional counterparts of these examples are $L^2$-functions which take only non-negative values or values in $[0,1]$.
Our only assumption on the operator $T$ is that it is Lipschitz continuous with some constant $L>0$, meaning
\begin{align}\label{ineq:lipschitz}
    \norm{T(u)-T(v)}\leq L\norm{u-v},\quad u,v\in U.
\end{align}

\begin{prop}
Under the assumptions \cref{eq:scaling_U} and \cref{ineq:lipschitz}, operator $T$ has an eigenvector, i.e., there exists $u^*\in U$ and $\lambda>0$ such that $T(u^*)=\lambda u^*$.
\end{prop}

\begin{proof}
The proof uses Banach's fixed point theorem. If $T$ is Lipschitz continuous with constant $L<1$, then $T$ is a contraction and hence has a unique fix point $u^*\in U$ satisfying $T(u^*)=u^*$.

If the Lipschitz constant $L$ of $T$ is greater or equal than $1$, we define a new map 
$T_\eps:=T/(L+\eps)$ for $\eps>0$, which is a contraction. Furthermore, $T_\eps$ maps $U$ into $U$ by the assumption that $cU\subset U$ for $0<c\leq 1$. Hence, reasoning as above, there exists a unique $u^*\in U$ such that $T_\eps(u^*)=u^*$, which can be rewritten as $T(u^*)=(L+\eps)u^*$.
\end{proof}

\begin{rem}
The main drawback of the  above result is that one cannot ensure that the eigenvector, the existence of which is proved, is different from zero.
If it is known that $T(0)=0$, \revision{for instance, if $T$ is a linear map,} then $0$ is the unique fixed point of $T_\eps$ \revision{whereas non-trivial eigenvectors can exist also in this case.}
\end{rem}

\revision{%
\begin{rem}[Existence of normalized eigenvectors]
As in the linear setting, the existence of normalized eigenvectors is not ensured.
As counterexample, one can simply let $T$ coincide with a rotation which possesses no real non-zero eigenvector.
For a less trivial example, one can define the nonlinear operator $T:\R^2\to\R^2$ as
$$T(u) = \abs{1-\norm{u}}(-u_2,u_1) + \norm{u}(u_1,u_2).$$
The eigenvectors of this operator are given by $\{u\in\R^2\st\norm{u}=1\}$.

If the operator $T$ is positively homogeneous of some degree $p$, i.e., $T(cu)=c^p\,T(u)$ for $c>0$, and a non-zero eigenvector exists, it can obviously be normalized.
\end{rem}}

Hence, in order to show that a given neural net $T$ has an eigenvector, we simply have to make sure $T$ is Lipschitz continuous.
This, however, is fulfilled for most networks types.
\begin{example}[Deep neural nets]
Deep neural nets of the form
\begin{align}\label{eq:deep_net}
    T(u)=\sigma(A^{(n)}\dots\sigma(A^{(2)}\sigma(A^{(1)}u+b^{(1)})+b^{(2)})\dots+b^{(n)}),\quad u\in\R^n,
\end{align}
with weight matrices $A^{(k)}\in\R^{n\times n}$ and bias vectors $b^{(k)}\in\R^n$ for $k=1,\dots,n$ are Lipschitz continuous if the activation function $\sigma$ is Lipschitz continuous.
For most popular choices of $\sigma$ (such as $\relu$, $\mathrm{TanH}$, $\mathrm{Logistic}$, $\mathrm{SoftPlus}$, etc.) this is fulfilled.
\end{example}

\subsection{The Kernel of ReLU Networks}
We now study the kernel of $T$, which is given by
\begin{align}\label{eq:kernel}
    \mathrm{ker}(T):=\{u\in\H \st T(u)=0\}.
\end{align}
In the following we study the case where $T$ is a single-layer neural network with $\relu$ activation and sketch how to extend this for multi-layer networks.
The fundamental difference between these two cases is that, in general, the kernel is a convex cone for single-layer networks and, more generally, a convex polyhedron in the multi-layer case. 
\subsubsection{Single-layer case}
We consider a single-layer network of the form 
\begin{align}
    T(u)=\sigma(Au+b),\quad u\in\R^n,
\end{align}
where $A\in\R^{n\times n}$ is a square weight matrix, $b\in\R^n$ denotes a bias vector, and $\sigma$ is some activation function with $\sigma(x)=0$ for $x\leq 0$, the prototypical example being
\begin{align}
    \sigma(x)=\relu(x)=\max(x,0)
\end{align}
or any smoothed version of it. 
Hence, the kernel of $T$ can be written as
\begin{align}\label{eq:RELU_kernel}
    \mathrm{ker}(T)=\{u\in\R^n\st Au+b\leq 0\},
\end{align}
where the inequality should be understood component-wise.
We will make one assumption on the weights and biases which allows us to characterize the kernel explicitly. 
Without this assumption weaker versions of our results remain true.
\begin{assumption}[Range condition]\label{ass:range_A}
We assume that $b\in\ran(A)$.
\end{assumption}
For the following statements, we need the notion of a convex cone with tip.
\begin{definition}
A set $C\subset\R^n$ is called convex cone with tip $v_0\in\R^n$ if $u+\alpha(u-v_0)\in C$ for all $u\in C$ and $\alpha\geq 0$. $C$ is called polyhedral if it can be written as $C=\{u+v_0\in\R^n\st Mu\geq 0\}$ with some matrix $M\in\R^{n\times n}$.
\end{definition}
The following lemma states that preimages under affine maps preserve convex cones with tips and polyhedrality.
\begin{lemma}
Let $C\subset\R^n$ be a convex cone with tip $v_0\in\R^n$ and $F:\R^n\to\R^n,\; u\mapsto Au+b$, be an affine map. If there is $u_0\in\R^n$ meeting $F(u_0)=v_0$, then $D:=\preim_F(C)$ is a convex cone with tip $u_0$. Furthermore, if $C=\{u+v_0\in\R^n\st Mu\geq 0\}$ is polyhedral, so is $D$ and it holds $D=\{u+u_0\in\R^n\st MAu\geq 0\}$.
\end{lemma}
\begin{proof}
We take an element $u\in D=\preim_F(C)$, meaning that $F(u)\in C$. We have to show that $u_\alpha:=u+\alpha(u-u_0)\in D$ for any $\alpha\geq 0$. To this end we compute
\begin{align*}
    F(u+\alpha(u-u_0))&=A(u+\alpha(u-u_0))+b=Au+\alpha Au-\alpha Au_0+b\\
    &=Au+b+\alpha(Au+b-(Au_0+b))=F(u)+\alpha(F(u)-F(u_0))\\
    &=F(u)+\alpha(F(u)-v_0)\in C,
\end{align*}
which follows since $F(u)\in C$ and $C$ is a cone with tip $v_0$. Hence, we have established $u_\alpha\in D$. 
For the second statement we assume that $C$ is polyhedral and obtain
\begin{align*}
    D&=\preim_F(C)=\{u\in\R^n\st F(u)\in C\}=\{u\in\R^n\st F(u)=v+v_0,\, Mv\geq 0\}\\
    &=\{u+u_0\in\R^n\st Au+F(u_0)=v+v_0,\, Mv\geq 0\}\\
    &=\{u+u_0\in\R^n\st MAu\geq 0\},
\end{align*}
where we again used $F(u_0)=v_0$. This shows that $D$ is polyhedral and concludes the proof. 
\end{proof}
Applying these insights to the kernel \cref{eq:RELU_kernel} of the $\relu$ network $T$, one obtains
\begin{thm}[Kernel of a single-layer $\relu$ network]\label{thm:kernel_of_single_layer}
Under \cref{ass:range_A} the kernel of $T$ is a polyhedral convex cone with tip $u_0$ where $Au_0=-b$. Furthermore, it holds
\begin{align}
    \ker(T)=\{u+u_0\in\R^n\st -Au\geq 0\}.
\end{align}
\end{thm}
\begin{proof}
For the proof one applies the statements above to the affine map $F(u)=Au+b$ and the polyhedral convex cone $C=\{u\in\R^n\st u\leq 0\}$ with tip $v_0=0$.
This cone can be written as $C=\{u\in\R^n\st Mu\geq 0\}$ where $M:=-\mathbb{I}$ denotes the negative identity matrix.
\end{proof}

\begin{rem}[Other activation functions]
If one considers activation functions which fulfill $\sigma(x)=0$ if and only if $x=0$, the discussion of the kernel becomes trivial. Either the equation $Au=-b$ has at least one solution, in which case the kernel is a single point or an affine space, or it does not, in which case the kernel is empty.
\end{rem}

\subsubsection{Multi-layer case}
Now we sketch how the kernel of a deep network can be obtained. For simplicity, we restrict ourselves to the case of two layers and consider
\begin{align}
    T(u)=\sigma(F^{(1)}(\sigma(F^{(2)}(u)))),\quad u\in\R^n,
\end{align}
where $F^{(k)}(u)=A^{(k)}u+b^{(k)}$ for $k=1,2$, denote the corresponding affine functions. 
We assume that there is an element $u_0$ which meets $A^{(2)}u_0=-b^{(2)}$ and obtain
\begin{align*}
    C^{(2)}:=\preim_{F^{(2)}}(\R^n_-)=\{u+u_0\in\R^n\st -A^{(2)}u\geq 0\}
\end{align*}
from \cref{thm:kernel_of_single_layer}. 
This implies
\begin{align*}
    \ker(T)&=\{u\in\R^n\st F^{(2)}(\sigma(F^{(1)}(u)))\leq 0\}=\{u\in\R^n\st \sigma(F^{(1)}(u))\in C^{(2)}\}.
\end{align*}
At this point one cannot simply take the preimage of $C^{(2)}$ under $F^{(1)}$ to obtain the kernel of $T$, since the activation function $\sigma$ is in the way. 
However, for $\relu$-type activation functions one can simplify this to 
\begin{align}
    \ker(T)=\preim_{F^{(1)}}\left(C_+^{(2)}\right)
\end{align}
where $C_+:=\{\max(u,0)\st u\in C\}$ denotes the positive part of a set $C\subset\R^n$.
Note than one can write $C_+^{(2)}$ as intersection of two polyhedral cones $C_+^{(2)}=C^{(2)}\cap\R^n_+$, which is a polyhedron, i.e., an intersection of finitely many half-spaces.
Hence, the kernel is given by the preimage of the polyhedron $C_+^{(2)}$ under the affine map $F^{(1)}$ which is again a polyhedron, according to \cite{zhang2012polyhedron}.
We condense these insights into 
\begin{thm}[Kernel of a multi-layer $\relu$ network]
Let $T:\R^n\to\R^n$ be a multi-layer neural net with $\relu$ activation, given by \cref{eq:deep_net}.
Then $\ker(T)$ is a (possibly unbounded) convex polyhedron.
\end{thm}

\begin{rem}[Deep networks with multiple layers]
Note that for deep networks the kernel is found by taking successive preimages of non-negative polyhedra under affine maps and keeping only the positive parts. 
Hence, the kernel will in general be smaller the larger the numbers of layers is.
In particular, if one of the preimages does not intersect the positive orthant, the kernel will be empty.
\end{rem}

\section{Analysis of the Nonlinear Power Method}
As a first step towards computing eigenvectors of neural nets, we study \cref{alg:simple} below, which is a straightforward generalization of \cref{alg:lin} to the nonlinear case, first studied in \cite{HAITFRAENKEL2021103041}. 
We first repeat some key results from~\cite{HAITFRAENKEL2021103041} in \Cref{Sec::simple_power_it}.
Then, in \Cref{Sec::proximal_operators_analysis} we analyze the algorithm for a specific family of nonlinear denoisers, which are proximal operators based on convex, absolutely one-homogeneous regularizers (such as total variation \cite{burger2013guide} and total generalized variation \cite{bredies2010total}).
We will prove that the proximal power method converges to an eigenvector under natural assumptions \revise{and give a convergence rate for the angle (cf.~\cref{eq:Theta}).}
Note that these homogeneous regularizers are not sensitive to the range of the vectors.
In contrast, in \Cref{Sec::toy_example} we will present a toy example of a non-homogeneous single-layer $\relu$ net, showing the limitations of the standard power method and motivating the generalized method defined in \Cref{Sec::adapted}.

\subsection{A Simple Nonlinear Power Method}
\label{Sec::simple_power_it}
We define the following nonlinear power-iteration algorithm, to which we refer as a \emph{simple} algorithm. This is an immediate generalization of \cref{alg:lin}, replacing the linear matrix $L$ by a nonlinear operator $T(\cdot)$.  

\begin{algorithm}
\caption{{\it Simple} Power Method with Generic Operator $T$}
\label{alg:simple}
~~ \\
\textbf{Input:} $f\in \H$, $\varepsilon>0$.
\begin{enumerate}
\item Initialize: $u^0 \gets f/\|f\|$, $\,\,k \gets 1.$
    \item Repeat until $\|u^{k+1}-u^k\|< \varepsilon$:\\ 
    $u^{k+1} \gets \frac{T(u^k)}{\|T(u^k)\|}$, $\,\,\,k \gets k+1.$
\end{enumerate}
\textbf{Output:} $(u^*, \lambda^*)$, where $u^*=u^k$, $\lambda^*=R(u^*)$, with $R$ defined in \cref{eq:Rayleigh}.
\end{algorithm}

For \cref{alg:simple} to be well defined, we assume that for all $k\in\N$ it holds $T(u^k)\neq 0$.
For proximal operators this can be shown to be true (cf.~\Cref{Sec::proximal_operators_analysis} below).
In the following propositions we recall some key properties of \cref{alg:simple} from \cite{HAITFRAENKEL2021103041}.

\begin{prop}
\cref{alg:simple} converges after a finite number of steps, i.e, there is $N\in\N$ such that for all $k>N$ it holds $u^{k+1}=u^k$ \textbf{if and only if} $u^k$ solves the eigenproblem \cref{eq:EV}.
\end{prop}
\begin{prop}\label{prop:Ray}${ }$\\
\begin{minipage}[t]{\textwidth}
\begin{enumerate}
    \item For every $k\in\N$, $|R(u^k)| \leq \|T(u^k)\|$. This holds in equality \textbf{if and only if} $u^k$ solves the eigenproblem \cref{eq:EV}.
    \item If exactly at iteration $k=N$ \cref{alg:simple} converged, then for all $k<N, |R(u^k)|<\norm{T(u^k)}$, and for all $k\geq N, |R(u^k)|=\norm{T(u^k)}$. 
\end{enumerate}
\end{minipage}
\end{prop}


\begin{prop}\label{prop:theta}
The angle between $u^k$ and $T(u^k)$ is $\pi n$ for $n \in \Z$ \textbf{if and only if} $u^k$ solves the eigenproblem \cref{eq:EV}.
\end{prop}

In our numerical experiments a Rayleigh quotient which reaches a constant value serves as a good indication for convergence to an eigenvector. 
For the operators tested, the Rayleigh quotient approximately increases to the eigenvalue, however, a general proof for this is pending.
As discussed in \Cref{Sec::def}, we also aim to reach an angle \cref{eq:Theta} close to $0$, which will serve as our validation measure. 
Hence, we define
\begin{definition}\label{def::approx_eigen}
We call $u$ an approximate (positive) eigenvector of $T$ if the angle $\theta$ given by \cref{eq:Theta} meets $0 < \theta < 0.5^\circ$ ($1^\circ=\pi/180$). 
\end{definition}




We now examine the behavior of \cref{alg:simple} in two different cases. 
On one hand, we prove its convergence to a non-trivial eigenvector for proximal operators.
On the other hand, we present a very simple toy example for the far more complex, non-homogeneous neural network, for which \cref{alg:simple} is not able to produce meaningful eigenvectors.

\subsection{Analysis of a Proximal Power Method}
\label{Sec::proximal_operators_analysis}
In this section we analyse a nonlinear power method associated to the proximal operator of a convex functional. Let $J:\H\to\R\cup\{\infty\}$ be an absolutely one-homogeneous, convex, and lower semi-continuous functional, defined on a Hilbert space $\H$.
Absolute one-homogeneity means that for all $u\in\H$ and numbers $c\neq 0$
\begin{align}
    J(cu)=|c|J(u),\quad J(0)=0.
\end{align}
For detailed properties of such functionals see \cite{bungert2019nonlinear,burger2016spectral}.
The proximal operator of $J$ is given by
\begin{align}\label{eq:prox}
    \prox_\alpha^J(u):=\argmin_{v\in\H}\frac{1}{2}\norm{v-u}^2+\alpha J(v),
\end{align}
where $f\in\H$ and $\alpha\geq 0$ denotes the regularization parameter. 
A prototypical example for $J$ is given by the total variation, defined on the Hilbert space $\H=L^2(\Omega)$ by setting 
\begin{align}\label{eq:tv}
    J(u)=
    \sup\left\lbrace\int_\Omega u\div\phi\d x\st \phi\in C^\infty_c(\Omega),\,\sup_{x\in\Omega}|\phi(x)|\leq 1\right\rbrace,\quad u\in L^2(\Omega).
\end{align}
In this case the proximal operator \cref{eq:prox} coincides with the solution of the famous Rudin-Osher-Fatemi model \cite{rudin1992nonlinear}.
We will now analyze \cref{alg:simple} with the nonlinear operator
\begin{equation}\label{eq:Tu_prox}
    T(u) = \prox_{\alpha(u)}^J(u),
\end{equation}
which gives the \emph{proximal power method}
\begin{align}\label{eq:proximal_power_it}
\begin{cases}
    u^0&={f}/{\norm{f}},\\
    v^{k}&=\prox_{\alpha(u^k)}^J(u^k),\quad k\geq 1,\\
    u^{k+1}&={v^{k}}/{\norm{v^k}}.
\end{cases}
\end{align}
Here $\alpha(u)$ denotes regularization parameters which are allowed to depend on $u$.
\revision{Hence, the operator $T$ in \eqref{eq:Tu_prox} is more general than a proximal operator. Indeed, its regularization parameter can be chosen adaptively to accelerate the convergence of the proximal power method \eqref{eq:proximal_power_it}, as our numerical results show.}
Constant parameters can be considered by setting $\alpha(u)\equiv\alpha>0$, \revision{in which case $T$ coincides with the standard proximal operator}.
Furthermore, $f\in\calN(J)^\perp$ where 
\begin{align}
    \calN(J)=\{u\in\H\st J(u)=0\}
\end{align}
is referred to as null-space of $J$ and indeed is a \revision{linear space \cite{benning2013ground,bungert2019solution}}.
An assumption which does not restrict generality but simplifies the notation is that one considers the proximal operator \cref{eq:prox} acting on $u\in\calN(J)^\perp$ only.
Here $\calN(J)^\perp$ denotes the orthogonal complement of the null-space.
This is due to the fact that it holds 
\begin{align}\label{eq:shift_invariance_prox}
    \prox_\alpha^J(u)=\prox_\alpha^J(u-\overline{u})+\overline{u},
\end{align}
where $\overline{u}\in\calN(J)$ denotes the orthogonal projection of $u\in\H$ onto $\calN(J)$.
Furthermore, $\prox_\alpha^J(u)\in\calN(J)^\perp$ if $u\in\calN(J)^\perp$.
If, for example, $J$ equals the total variation, then the null-space consists of constant functions and its orthogonal complement is given by all zero-mean functions.
In this case, the proximal operator is invariant under the mean value $\overline{u}$ of its input~$u$.

To show convergence of power method associated to the operator $T$ in \cref{eq:Tu_prox}, we need two assumptions on the interplay between the functional $J$ and the Hilbert norm $\norm{\cdot}$.
\begin{assumption}[Poincar\'{e} inequality]\label{ass:poincare}
There is a constant $c_P>0$ such that $c_P\norm{u}\leq J(u)$ for all $u\in\calN(J)^\perp$.
\end{assumption}
\begin{assumption}[Compact sub-level sets]\label{ass:compact}
The sub-level sets of $\norm{\cdot}+J(\cdot)$ are compact.
\end{assumption}
\begin{example}
A relevant example where \cref{ass:poincare} and \cref{ass:compact} are fulfilled is $J(u)=\norm{\nabla u}_{p}$, where $p\in\left(\frac{2n}{n+2},\infty\right]$ and $\Omega\subset\R^n$ is a bounded Lipschitz domain with $n\geq 2$ (cf.~\cite{bungert2019asymptotic} and \cite{bungert2020structural} for $p=\infty$).
Note that in the relevant case that $J$ equals the total variation \cref{eq:tv} the assumptions hold true only in dimension $n=1$, since the compact embedding $\mathrm{BV}(\Omega)\Subset L^2(\Omega)$ exists only in one dimension.
In two dimensions the embedding is only continuous and in higher dimensions it does not even exist.
However, by demanding additional regularity for the initial condition $u^0$ of the power method, as for instance $u^0\in L^\infty(\Omega)$, one can show that \cref{ass:poincare} and \cref{ass:compact} hold true \emph{along} the sequence which is generated by the iteration (cf.~\cite{bungert2019nonlinear}).
\end{example}
\begin{example}
If $\H$ is finite-dimensional and $J$ is norm on a subspace of $\H$, then both assumptions are met due to the equivalence of norms in finite dimensions.
\end{example}
In the following, we will need an important result (see for instance~\cite{bauschke2011convex}) which characterizes the subdifferential of absolutely one-homogeneous functionals
\begin{prop}[Subdifferential]\label{prop:subdifferential}
Let $J:\H\to\R\cup\{\infty\}$ be convex and absolutely one-homogeneous.
Then its subdifferential in $u\in\H$ is given by 
\begin{align}
    \partial J(u)=\{p\in\H\st\langle p,v\rangle\leq J(v),\,\forall v\in\H,\,\langle p,u\rangle=J(u)\}.
\end{align}
\end{prop}

Our first result characterizes the maximal regularization parameter $\alpha(u)$ such that the proximal power method \cref{eq:proximal_power_it} is well-defined.
\begin{prop}\label{prop:extinction}
For $u\in \calN(J)^\perp$ it holds that $T(u)=\prox_{\alpha(u)}^J(u)=0$ if and only if $\alpha(u)\geq J_*(u)$ where
\begin{align}\label{eq:dual_norm}
    J_*(u):=\sup_{p\in \calN(J)^\perp}\frac{\langle u,p\rangle}{J(p)}\geq\frac{\norm{u}^2}{J(u)}.
\end{align}
Furthermore, equality holds in \cref{eq:dual_norm} if and only if $u$ is an eigenvector of $\partial J$, meaning $\lambda u\in\partial J(u)$ for $\lambda=J(u)/\norm{u}^2$.
\end{prop}
\begin{proof}
The first claim was proved in \cite{bungert2019solution} and we just show the second one. 
By choosing $p=u$ in the supremum one always has $J_*(u)\geq\norm{u}^2/J(u)$.
If $u$ is an eigenvector there is $\lambda\geq 0$ such that $\lambda u\in\partial J(u)$. 
Using \cref{prop:subdifferential} we can draw two conclusions. First of all, it holds $\langle\lambda u,u\rangle =J(u)$ and hence $\lambda=J(u)/\norm{u}^2$. 
Secondly, one has $\langle\lambda u,p\rangle\leq J(p)$ for all $p\in\H$ which implies
\begin{align*}
    J_*(u)=\sup_{p\in \calN(J)^\perp}\frac{\langle u,p\rangle}{J(p)}=\frac{1}{\lambda}\sup_{p\in \calN(J)^\perp}\frac{\langle\lambda u,p\rangle}{J(p)}\leq\frac{1}{\lambda}=\frac{\norm{u}^2}{J(u)},
\end{align*}
such that $J_*(u)=\norm{u}^2/J(u)$.
Conversely, if $J_*(u)=\norm{u}^2/J(u)$ then one obtains
\begin{align*}
    \langle\lambda u,v\rangle=\lambda J(v)\frac{\langle u,v\rangle}{J(v)}\leq\lambda J(v)J_*(u)=J(v),\quad\forall v\in\H,
\end{align*}
for $\lambda = J(u)/\norm{u}^2$. 
Hence, $\lambda u\in\partial J(u)$.
\end{proof}
\revision{%
\begin{rem}
The condition $\alpha(u)\geq J_*(u)$ can be equivalently formulated through $\frac{u}{\alpha(u)}\in\partial J(0)$ (see~\cite[Equation (2.10)]{bungert2019nonlinear}).
\end{rem}}
\begin{corollary}[Well-definedness of the proximal power method]\label{cor:well-defined}
Let $\alpha(u)$ satisfy
\begin{align}
    \alpha(u)<J_*(u)
\end{align}
Then for all initial conditions $f\in\calN(J)^\perp$ the proximal power method \cref{eq:proximal_power_it} is well-defined.
\end{corollary}
\begin{proof}
By \cref{prop:extinction} we know that $T(u^k)=\prox_{\alpha(u^k)}^J(u^k)\neq 0$ if and only if $\alpha(u^k)<J_*(u^k)$. 
Hence $v^{k}\neq 0$ in \cref{eq:proximal_power_it} which makes the iteration well-defined.
\end{proof}
Next, we show that the functional values $J(u^k)$ decrease along the proximal power method~\cref{eq:proximal_power_it}.
This will be the key ingredient for convergence of the algorithm and follows from
\begin{prop}[Decrease of the Rayleigh quotient]\label{prop:decrease_rayleigh}
Let $u\in\calN(J)^\perp\setminus\{0\}$, $v:=\prox_{\alpha(u)}^J(u)$, and $\alpha(u)<J_*(u)$. Then it holds
\begin{align}\label{ineq:decrease}
    \frac{J(v)}{\norm{v}}\leq\frac{J(u)}{\norm{u}}
\end{align}
with equality if and only if $\lambda u\in\partial J(u)$ for $\lambda=J(u)/\norm{u}^2$.
\end{prop}
\begin{proof}
The optimality of $v$ means that 
\begin{align*}
    \frac{1}{2}\norm{v-u}^2+\alpha(u)J(v)\leq\frac{1}{2}\norm{w-u}^2+\alpha(u)J(w),\quad\forall w\in\H,
\end{align*}
with equality if and only if $w=v$, due to the strict convexity of the objective functional.
We define $w:=\frac{\norm{v}}{\norm{u}}u$ which meets
\begin{align*}
    \norm{w-u}^2=\norm{w}^2-2\langle w,u\rangle+\norm{u}^2
    =\norm{v}^2-2\norm{v}\norm{u}+\norm{u}^2
    \leq\norm{v-u}^2.
\end{align*}
Plugging this into the optimality above yields
\begin{align*}
    \alpha(u)J(v)\leq\alpha(u)J(w)=\alpha(u)\frac{\norm{v}}{\norm{u}}J(u).
\end{align*}
From \cref{prop:extinction} we infer that $v\neq 0$.
Hence, we can divide by $\norm{v}$, cancel $\alpha(u)$, and arrive at the assertion. 
Equality holds if and only if 
$$v=w=\frac{\norm{v}}{\norm{u}}u=cu$$
where $c:=\norm{v}/\norm{u}>0$. 
The optimality condition for problem \cref{eq:prox} is given by
$$\frac{u-v}{\alpha(u)}\in\partial J(v).$$
Plugging in the expression for $v$ yields
$$\frac{1-c}{\alpha(u)}u\in\partial J(u),$$
where we used that $\partial J(v)=\partial J(cu)=\partial J(u)$ since $c>0$ and $J$ is absolutely one-homogeneous (cf.~\cite{bungert2019nonlinear}).
Hence, we get $\lambda u\in\partial J(u)$ with $\lambda={(1-c)}/{\alpha(u)}={J(u)}/{\norm{u}^2}$.
\end{proof}
As mentioned above we can now prove that the power method decreases the energy $J$.
\begin{corollary}[Energy decrease]\label{cor:decrease_J}
Let $\alpha(u)<J_*(u)$.
Then the iterates of \cref{eq:proximal_power_it} with $f\in\calN(J)^\perp$ fulfill
\begin{align}
    J(u^{k+1})&\leq J(u^k),
\end{align}
with equality if and only if $\lambda u^k\in\partial J(u^k)$ for $\lambda=J(u^k)/\norm{u^k}^2$.
\end{corollary}
\begin{proof}
Applying \cref{prop:decrease_rayleigh} to $u=u^k$ and $v=\prox_{\alpha(u^k)}^J(u^k)$ yields
$$J(u^{k+1})=J\left(\frac{\prox_{\alpha(u^k)}^J(u^k)}{\norm{\prox_{\alpha(u^k)}^J(u^k)}}\right)=\frac{J(\prox_{\alpha(u^k)}^J(u^k))}{\norm{\prox_{\alpha(u^k)}^J(u^k)}}\leq\frac{J(u^k)}{\norm{u^k}}=J(u^k),$$
where we used the absolute one-homogeneity of $J$ and $\norm{u^k}=1$.
\end{proof}

Some remarks regarding the admissible choice of $\alpha(u)$ are in order.
\begin{rem}
For a quick algorithm one does not have to compute the dual norm $J_*(u)$, which bounds the admissible regularization parameters $\alpha(u)$, explicitly. 
Instead, one can make use of the lower bound $J_*(u)\geq\norm{u}^2/J(u)$ which was derived in \cref{prop:extinction}.
\end{rem}
\begin{rem}[Constant regularization parameter]
Of course, choosing a constant regularization parameter in \cref{eq:proximal_power_it}, which does not depend on the previous iterate, is possible.
Let us choose $\alpha(u)\equiv\alpha<1/J(u^0)$ for all $u\in\H$. 
Then by \cref{cor:decrease_J} it holds that $\alpha(u^k)<1/J(u^0)\leq 1/J(u^k)$ for any $k\in\N$ and hence the proximal power method \cref{eq:proximal_power_it} is well-defined according to \cref{cor:well-defined}.
\end{rem}
Before we can prove our main theorem, the convergence of the proximal power method, we need a lemma which studies continuity properties of the operator \cref{eq:Tu_prox}.
\begin{lemma}[Continuity of the proximal operator]\label{lem:convergences}
Let $(u^k)\subset\calN(J)^\perp$ be a sequence converging to $u^*$, and let $v^k:=\prox_{\alpha(u^k)}^J(u^k)$ for $k\in\N$.
If the sequence of regularization parameters fulfills $\lim_{k\to\infty}\alpha(u^k)=\alpha^*>0$ then $(v^k)$ converges to some $v^*\in\H$ and it holds $v^*=\prox_{\alpha^*}^J(u^*)$.
\end{lemma}
\begin{proof}
From the optimality of $v^k$ we deduce
$$\frac{1}{2}\norm{v^k-u^k}^2+\alpha(u^k)J(v^k)\leq\frac{1}{2}\norm{v-u^k}^2+\alpha(u^k)J(v),\quad\forall v\in\H.$$
Choosing $v=0$, we can infer
\begin{align}\label{ineq:optimality_vk}
    \limsup_{k\to\infty}\frac{1}{2}\norm{v^k-u^k}^2+\alpha(u^k)J(v^k)\leq\limsup_{k\to\infty}\frac{1}{2}\norm{u^k}^2<\infty,
\end{align}
since $(u^k)$ is a convergent sequence and hence bounded.
By triangle inequality it holds
\begin{align*}
    \norm{v^k}\leq\norm{v^k-u^k}+\norm{u^k}
\end{align*}
which together with \cref{ineq:optimality_vk} shows that $\limsup_{k\to\infty}\norm{v^k}<\infty$.
Furthermore, since we have assumed that $\lim_{k\to\infty}\alpha(u^k)=\alpha^*>0$, we also get that  $\limsup_{k\to\infty}J(v^k)<\infty$.
Hence, by \cref{ass:compact} a subsequence of $(v^k)$ converges to some $v^*\in\H$.
Using lower semi-continuity of $J$ and the strong convergences of $(u^k)$ and $(v^k)$ it holds
\begin{align*}
    \frac{1}{2}\norm{v^*-u^*}^2+\alpha^*J(v^*)&\leq\liminf_{k\to\infty}\frac{1}{2}\norm{v^k-u^k}^2+\alpha(u^k)J(u^k)\\
    &\leq\liminf_{k\to\infty}\frac{1}{2}\norm{v-u^k}+\alpha(u^k)J(v)
    =\frac{1}{2}\norm{v-u^*}^2+\alpha^*J(v),\quad\forall v\in\H.
\end{align*}
This shows that $v^*=\prox_{\alpha^*}^J(u^*)$.
The same argument shows that in fact every convergent subsequence of $(v^k)$ converges to $v^*$ and hence the whole sequence $(v^k)$ converges to $v^*$.
\end{proof}

In the previous lemma we have seen that the choice of regularization parameters $\alpha(u)$ cannot be arbitrary but must be such that along a convergent sequence $(u^k)$ also $\alpha(u^k)$ converges to some positive value.
In the following, we thus fix two possible parameter choices which have this property and are feasible in practical applications.
\revision{%
\begin{definition}[Parameter rules]\label{def:parameter_rule}
Let $u^0\in\H$ with $\norm{u^0}=1$ and $0<J(u^0)<\infty$, and let $c\in(0,1)$ be a constant.
We say that $\alpha(u)$ follows a \emph{constant parameter rule} if $\alpha(u)=\frac{c}{J(u^0)}$, and a \emph{variable parameter rule} if $\alpha(u) \leq \frac{c}{J(u)}$.

In applications, one would like to choose $c$ large but smaller than one, e.g., $c=0.9$.
\end{definition}}
\revision{
\begin{example}[Rayleigh quotient minimizing flows \cite{feld2019rayleigh}]\label{ex:raleigh_feld}
In \cite{feld2019rayleigh} (see also \cite{benning2017learning}) a flow to minimize the Rayleigh quotient $R(u)=J(u)/H(u)$ was studied. Its time discretization with time step size $\tau>0$ is
\begin{align*}
    \begin{cases}
    v^k = \argmin \frac{1}{2\tau}\norm{u-u^k}^2 - R(u^k)\langle q^k,u\rangle + J(u),\quad q^k\in\partial J(u^k),\\
    u^{k+1} = \frac{v^k}{\norm{v^k}}.
    \end{cases}
\end{align*}
For the case that $H(u)=\norm{u}$ equals the Hilbert norm one observes
\begin{align*}
    &\phantom{=}\;\;\argmin \frac{1}{2\tau}\norm{u-u^k}^2 - R(u^k)\langle q^k,u\rangle + J(u) \\
    &= \argmin \frac{1}{2\tau}\norm{u-u^k}^2 - J(u^k)\langle u^k,u\rangle + J(u) \\
    &= \argmin \frac{1}{2\tau}\|u-\underbrace{(1+\tau J(u^k))}_{=:c^k}u^k\|^2 + J(u) \\
    &= \prox_{\tau}^J(c^ku^k) \\
    &= c^k \prox_{\tau/c^k}^J(u^k).
\end{align*}
Hence, the method in \cite{feld2019rayleigh} is a special case of the proximal power method \eqref{eq:proximal_power_it} with step size $\alpha(u)=\frac{\tau}{1+\tau J(u)}$.
For $\tau\to\infty$ this converges from below to the maximal upper bound $\frac{1}{J(u)}$ of the variable parameter rule in \cref{def:parameter_rule}.
Hence, for a finite time step $0<\tau<\infty$ the Rayleigh quotient minimizing flows can be expected to have slower convergence properties compared to the proximal power method with rule $\alpha(u)=c/J(u)$, which is also reflected in our numerical comparison in \Cref{Sec::proximal_operators}.
\end{example}
}

\revision{%
The following lemma provides a rate of convergence of the angle between $u^k$ and $v^k=\prox_{\alpha(u^k)}^J(u^k)$ as $k\to\infty$. 
This is a crucial ingredient for proving convergence of the proximal power method.
Note that for the proof of the Lemma we assume that $(v^k)$ is uniformly bounded which will turn out to be satisfied for the proximal power method \eqref{eq:proximal_power_it}.
}
\begin{lemma}[Rate of asymptotic collinearity]\label{lem:collinearity}
Let $(v^k)$ and $(u^k)$ be given by the proximal power method \eqref{eq:proximal_power_it} and assume that $\norm{v^k}$ is uniformly bounded.
Under \cref{ass:poincare} it holds
\revise{%
\begin{align}\label{eq:collinearity}
    \norm{v^k}-\langle v^k,u^k\rangle=o\left(\frac{1}{k}\right),\quad k\to\infty.
\end{align}
}
\end{lemma}
\begin{proof}
To show \cref{eq:collinearity}, we note that by definition of $v^k=\prox_{\alpha(u^k)}^J(u^k)$ it holds
\begin{align*}
    \frac{1}{2}\norm{v^k-u^k}^2+\alpha(u^k)J(v^k)\leq\frac{1}{2}\norm{v-u^k}^2+\alpha(u^k)J(v),\quad\forall v\in\H.
\end{align*}
Choosing $v=\norm{v^k}u^k$, expanding the squared norms, and using $\norm{u^k}=1$ yields
\begin{align*}
    \frac{1}{2}\norm{v^k}^2-\langle v^k,u^k\rangle+\frac{1}{2}+\alpha(u^k)J(v^k)\leq\frac{1}{2}\norm{v^k}^2-\norm{v^k}+\frac{1}{2}+\alpha(u^k)\norm{v^k}J(u^k).
\end{align*}
This can be simplified to
\begin{align*}
    \norm{v^k}-\langle v^k,u^k\rangle\leq \alpha(u^k)\norm{v^k}\left(J(u^k)-\frac{1}{\norm{v^k}}J(v^k)\right).
\end{align*}
First we note $\alpha(u^k)\norm{v^k}$ is uniformly bounded by some $C>0$. 
This is due to the fact that $\norm{v^k}$ is bounded by assumption and that we have the estimate
$$\alpha(u^k)\leq\frac{c}{J(u^k)},$$
where $J(u^k)$ is uniformly bounded away from zero thanks to \cref{ass:poincare}.
Using that by definition $u^{k+1}=v^k/\norm{v^k}$, we obtain
\begin{align*}
    \norm{v^k}-\langle v^k,u^k\rangle\leq C\left(J(u^k)-J(u^{k+1})\right).
\end{align*}
\revise{%
We obtain by telescopic summation:
\begin{align*}
    \sum_{k=0}^\infty\norm{v^k}-\langle v^k, u^k\rangle \leq C\left(J(u^0)-J(u^*)\right).
\end{align*}
Hence, the sequence on the left hand side is non-negative and summable, which finally yields  
\begin{align*}
    \norm{v^k}-\langle v^k,u^k\rangle=o\left(\frac{1}{k}\right),\quad k\to\infty.
\end{align*}
}
\end{proof}

Now we are ready to prove convergence of the proximal power method using constant or variable regularization parameters.

\begin{thm}[Convergence of the proximal power method]\label{thm:convergence_proximal}
Let $(u^k)\subset\H$ be the sequence generated by the proximal power method \cref{eq:proximal_power_it} with \emph{constant} or \emph{variable parameter rule} and initial datum $f\in\calN(J)^\perp$.
Under \cref{ass:poincare} and \cref{ass:compact} a subsequence of $(u^k)$ converges to some $u^*\in\H\setminus\{0\}$ which meets $\lambda u^*=\prox_{\alpha(u^*)}^J(u^*)$ for some $0<\lambda<1$.
\end{thm}
\begin{proof}
Since by assumption the sublevel sets of $\norm{\cdot}+J(\cdot)$ are compact and according to \cref{cor:decrease_J} sequence $(u^k)$ fulfills
$$\sup_{k\in\N}\norm{u^k}+J(u^k)\leq 1+J(u^0)<\infty,$$
it admits a convergent subsequence (which we do not relabel).
We denote by $u^*$ its limit and note that it fulfills 
\begin{align*}
    \norm{u^*}=1,\qquad
    J(u^*)\leq\liminf_{k\to\infty}J(u^k),
\end{align*}
by lower semi-continuity of $J$.
For the constant parameter rule it is trivial that $\alpha(u^k)$ converges to some positive value.
Let us therefore study the variable parameter rule $\alpha(u^k):=c/J(u^k)$ with $0<c<1$. 
Then according to \cref{cor:well-defined} it holds that $\alpha(u^k)$ is an increasing sequence which is bounded by $\alpha(u^k)\leq\frac{c}{c_P}$, where $c_P$ denotes the Poincar\'{e}-type constant from \cref{ass:poincare}. 
Hence, $\lim_{k\to\infty}\alpha(u^k)=\alpha^*$ exists and by lower semi-continuity of $J$ one has
\begin{align}\label{ineq:estimate_alpha*}
    \alpha^*=\lim_{k\to\infty}\alpha(u^k)=\frac{c}{\liminf_{k\to\infty}J(u^k)}\leq\frac{c}{J(u^*)}.
\end{align}

Applying \cref{lem:convergences} gives that $v^k:=\prox_{\alpha(u^k)}^J(u^k)$ converges to some $v^*\in\H$ and it holds $v^*=\prox_{\alpha^*}^J(u^*)$.
Note that \cref{ineq:estimate_alpha*} together with \cref{prop:extinction} implies that $v^*\neq 0$.

It remains to be shown that $v^*$ and $u^*$ are collinear since this implies that $u^*$ is an eigenvector.
To this end, we note that by \cref{lem:collinearity} it holds
\begin{align*}
    \lim_{k\to\infty}\norm{v^k}-\langle v^k,u^k\rangle=0.
\end{align*}
Using $v^k\to v^*$ and $u^k\to u^*$ we get $\norm{v^*}=\langle v^*,u^*\rangle$ which readily implies $u^*={v^*}/\norm{v^*}$.
Hence, we infer
$$\prox_{\alpha^*}^J(u^*)=v^*=\norm{v^*}u^*=\tilde{\lambda} u^*$$
with $\tilde{\lambda}:=\norm{v^*}>0$.
Note that $\tilde{\lambda}\leq 1$ since the proximal operator has unitary Lipschitz constant and thus meets
$$\tilde{\lambda}\norm{u^*}=\norm{\prox_{\alpha^*}^J(u^*)-\prox_{\alpha^*}^J(0)}\leq\norm{u^*}.$$
Indeed, it even holds $\tilde{\lambda}<1$ since otherwise the optimality condition
\begin{align}\label{eq:OC}
    \frac{1-\tilde{\lambda}}{\alpha^*}u^*\in\partial J(u^*)
\end{align}
of problem \cref{eq:prox} would reduce to $0\in\partial J(u^*)$ which means $J(u^*)=0$.
Due to $\norm{u^*}=1$ and \cref{ass:poincare} this is impossible.
It remains to be shown that it also holds $\lambda u^*=\prox_{\alpha(u^*)}^J(u^*)$ for a suitable $\lambda\in(0,1)$.
For the constant parameter rule this is trivially true with $\lambda=\tilde{\lambda}$.
In general, we rewrite \cref{eq:OC} as $\mu u^*\in\partial J(u^*)$ where $\mu:=(1-\tilde{\lambda})/\alpha^*$.
Then, $\prox_{\alpha(u^*)}^J(u^*)$ can be explicitly computed (see~\cite{benning2013ground}, for instance):
$$\prox_{\alpha(u^*)}^J(u^*)=(1-\alpha(u^*)\mu)_+u^*=\lambda u^*$$
with $\lambda:=(1-\alpha(u^*)\mu)_+\leq\tilde{\lambda}<1$.
Since $\alpha(u^*)=c/J(u^*)<J_*(u^*)$, \cref{prop:extinction} tells us that $\lambda>0$.
\end{proof}

\begin{rem}[Uniqueness and positivity]
For practical applications it is of great importance to understand to which eigenvector the proximal power method~\cref{eq:proximal_power_it} converges. 
While for the linear power method \cref{alg:lin} it is well-known that it converges to the largest eigenvector of $L$ if the initial datum $f$ ``contains'' the direction of this eigenvector, analogous statements for the nonlinear case are unknown, in general.
Since nonlinear eigenvectors are in general neither orthogonal nor do they form a basis \cite{benning2013ground}, the linear techniques do not apply. 
This also applies to all the methods which were referred to in \Cref{Sec::intro}.
However, there exists one scenario were one can actually say something about the limit of the power method, namely when positivity is involved.
If $f\geq 0$, where the partial order $\geq$ makes $\H$ a Banach lattice, then one can show that $u^k\geq 0$ for all $k\in\N$. 
Consequently, the limit $u^*$ is a non-negative eigenvector which can often  be characterized as the unique eigenvector of the proximal operator with the largest possible eigenvalue (see for instance \cite{kawohl2006positive,bungert2020structural}).
Analogous statements for rescaled gradient flows were shown in \cite{bungert2019asymptotic}.
Furthermore, in \cite{boyd1974power} similar statements were proven for nonlinear power iterations related to the computation of matrix norms (see also \cite{gautier2020computing}).
\end{rem}

\revision{%
\begin{rem}[Moreau Decomposition and Duality]
A fundamental property of proximal operators is the Moreau decomposition, which states that
\begin{align}\label{eq:moreau}
    u = \alpha \prox^{J^*}_{1/{\alpha}}\left(\frac{u}{\alpha}\right) + \prox^J_{\alpha}(u),\quad\forall u\in\H,
\end{align}
where $J^*(p)=\sup_{u\in\H}\langle p,u\rangle - J(u)$ is the Fenchel conjugate of $J$.
In particular, this means that $u$ satisfies $\lambda u = \prox_\alpha^J(u)$, if and only if $v:=u/\alpha$ satisfies $\mu v=\prox_{{1}/{\alpha}}^{J^*}(v)$ with $\mu=1-\lambda$.
Hence, the eigenvalue problems of $\prox_\alpha^J$ and $\prox_{{1}/{\alpha}}^{J^*}$ are equivalent.
If $J$ is absolutely one-homogeneous, its Fenchel conjugate equals the characteristic function of the set $K:=\partial J(0)=\left\lbrace p\in\H\st\langle p,v\rangle\leq J(v),\;\forall v\in\H\right\rbrace$ and its proximal operator reduces to a projection:
\begin{align}
    \prox_{1/\alpha}^{J^*}(v) = \proj_{K}(v).
\end{align}
Since large eigenvalues of this projection operator correspond to small eigenvalues of $\prox_\alpha^J$ and hence to large eigenvalues of the subdifferential operator $\partial J$, \Cref{Sec::def} already indicates that applying a power method to $\proj_{K}$ is not meaningful.
Indeed, if one considers the iteration
\begin{align}
    \begin{cases}
    v^{k} = \proj_{\beta K}(u^k), \\
    u^{k+1} = \frac{v^k}{\norm{v^k}},
    \end{cases}
\end{align}
where $\beta>0$ denotes some radius, it is obvious that $u^k=u^0$ for all $k\in\N$ whenever $B_1:=\lbrace u\in\H\st\norm{u}\leq 1\rbrace\subset\beta K$.
Hence, one needs to choose $\beta$ such that $\beta K\subset B_1$, which is equivalent to $\beta \norm{u}\leq J_*(u)$ for all $u\in\calN(J)^\perp$.
Here $J_*$ denotes the dual semi norm \eqref{eq:dual_norm}.
Together with the Poincar\'{e}-type inequality from \cref{ass:poincare} this would mean that $J_*$, $\norm{\cdot}$, and $J$, are equivalent norms on $\calN(J)^\perp$ which is excluded in almost all interesting applications.
\end{rem}}



\subsection{Networks: A Toy Example}
\label{Sec::toy_example}
Here we study an extremely simple ``2-pixel'' network with $\relu$ activation function. It is shown that nonlinearity plays a crucial role in such networks, where the operating range is critical. Unlike the one-homogeneous case of previous section, applying \cref{alg:simple} directly results in reaching only degenerate solutions.
This motivates us in the next section to develop a range-aware algorithm, which fits better to nonlinear operators, working only within an expected range.

For the linear eigenproblem \cref{eq:EV_lin} it is clear that, by linearity, every multiple of $u$ is an eigenvector with eigenvalue $\lambda$, as well.
For nonlinear maps, such as neural nets,  this is typically not the case, as the following example shows.
Let us consider the map $T:\R^2\to\R^2$, given by a simple one-layer $\relu$ network
$$T(u)=\relu\left[
\begin{pmatrix}
-1 & 0\\
0 & 1
\end{pmatrix}
\begin{pmatrix}
u_1\\
u_2
\end{pmatrix}
+\begin{pmatrix}
1\\
-1
\end{pmatrix}
\right]=
\begin{pmatrix}
\max(1-u_1,0)\\
\max(u_2-1,0)
\end{pmatrix}.
$$
After some calculations one easily sees that the set of eigenvectors of $T$ consists of the four different sets with respective eigenvalues, given by
\begin{alignat*}{2}
    S_1&=\{(\alpha,0)\st 0<\alpha< 1\}, \quad&&\lambda_1=\frac{1-\alpha}{\alpha},\\
    S_2&=\left\lbrace\left(\frac{1}{1+\alpha},\frac{1}{1-\alpha}\right)\st 0<\alpha<1\right\rbrace,\quad&&\lambda_2=\alpha,\\
    S_3&=\{(\alpha,0)\st \alpha< 0\},\quad&&\lambda_3=\frac{1-\alpha}{\alpha},\\
    S_4&=\{(u_1,u_2)\st u_1\geq 1,\;u_2\leq 1\},\quad&&\lambda_4=0.
\end{alignat*}

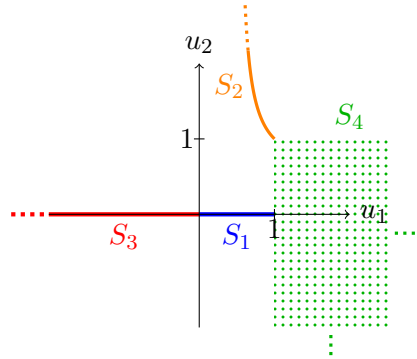
\begin{wrapfigure}{r}{0.45\textwidth}
    \centering
        \begin{tikzpicture}
        \draw[scale=1,domain=0:1,smooth,variable=\x,blue,ultra thick] plot ({\x},{0*\x}) ;
        \draw[scale=1,domain=-2:0,smooth,variable=\x,red,ultra thick]  plot ({\x},{0*\x}) ;
        \draw[scale=1,domain=-2.5:-2,smooth,variable=\x,red,ultra thick,dotted]  plot ({\x},{0*\x}) ;
        \draw[scale=1,domain=0.61:0.65,smooth,variable=\x,orange,very thick,dotted]  plot ({\x},{\x/(2*\x-1)}) ;
        \draw[scale=1,domain=0.65:1,smooth,variable=\x,orange,very thick]  plot ({\x},{\x/(2*\x-1)}) ;
        \fill [pattern=dots, pattern color=darkgreen] (1,1) rectangle (2.5,-1.5);
        \draw[darkgreen, very thick, dotted] (2.6,-0.25)--(2.9,-0.25);
        \draw[darkgreen, very thick, dotted] (1.75,-1.6)--(1.75,-1.9);
        \node[anchor=north,blue] at (.5,0) {$S_1$};
        \node[anchor=north,orange] at (.4,2) {$S_2$};
        \node[anchor=north,red] at (-1,0) {$S_3$};     
        \node[anchor=south,darkgreen] at (2,1) {$S_4$};  
        \draw[->] (-2,0) -- (2,0) node[right] {$u_1$};
        \draw[->] (0,-1.5) -- (0,2) node[above] {$u_2$};
        \draw(1,-2pt)--(1,2pt) node[below] {$1$};
        \draw(-2pt,1)--(2pt,1) node[left] {$1$};
        \end{tikzpicture}
    \vspace{-10pt}
    \caption{All eigenvectors of $T$. The dotted green region is the kernel.\label{fig:all_eigenvectors}}
\end{wrapfigure}
\cref{fig:all_eigenvectors} shows the sets of eigenvectors. So does the simple \cref{alg:simple} manage to find eigenvectors of $T$? 
Since by definition $T$ only returns non-negative vectors, one can never compute the negative eigenvectors in $S_3$.
Furthermore, due to normalization, one can only find eigenvectors with unit norm. 
The only non-negative eigenvector with unit norm is $(1,0)$, which lies in $S_4$, the set of eigenvectors with eigenvalue zero, i.e., the kernel of~$T$.
Indeed, one can check that for all initializations $u^0\notin S_4=\ker(T)$, \cref{alg:simple} converges to $(1,0)\in S_4$ in a finite number of steps. However, the more interesting eigenvectors in $S_1$, $S_2$ or $S_3$ cannot be reached, since the simple \cref{alg:simple} ignores the ``natural range'' where $T$ operates.

\section{Proposed Method for Range-Sensitive Operators}\label{Sec::adapted}
In the linear case, any eigenvector multiplied by a constant is an eigenvector with the same eigenvalue. In fact, for any homogeneous operator of degree $p\in [0,\infty)$ (i.e., $T(cu)=|c|^pu$), if $u$ is an eigenvector with eigenvalue $\lambda$, so is $cu$, $c\in \R$ with eigenvalue $\lambda |c|^{p-1}$. Thus, eigenvectors are not restricted to a certain range of values.
For more general nonlinear operators, however, the range of the vector is important. The operator and its respective eigenvectors may be range-sensitive. 
We shall now see how typical properties of images and denoisers imply that eigenvectors, as defined until now, exist only for unit eigenvalues. Thus, we later give a more relaxed definition. We discuss the finite-dimensional case $\H=\R^n$ since we aim at designing a power method for imaging purposes.
Hence, we view the denoising network as operator $T:U\to U$, where $U\subset\R^n$ meets \cref{eq:scaling_U}, and state two \revision{observations which apply to many denoising networks.} 
\begin{observation}[Pixels are in a specific value range]
Pixel values are in a specific predefined range, typically non-negative, and usually either $[0,1]$ or $[0,255]$. A nonlinear operator $T$ stemming from a neural network is usually trained on images within a predefined range. Thus at inference unexpected results may be produced, if the input image is not in the expected range. 
\end{observation}

\begin{observation}[Denoisers are unbiased] \label{ass:unbiased}
Common denoisers are designed for additive white noise (either Gaussian or uniform) of expected value zero. Thus it is common that a denoiser $T$ preserves the mean image value, which can be expressed as $\langle T(u),1 \rangle = \langle u,1 \rangle$.
\end{observation}

The second observation restricts possible eigenvalues of the denoiser to one, \revision{if one assumes non-negative inputs}:
\begin{prop}
For any vector $u\neq 0$ with non-negative entries and a denoiser $T$ admitting \cref{ass:unbiased} above, if $u$ meets $\lambda u = T(u)$, then $\lambda=1$.
\end{prop}
\begin{proof}
If $u$ is an eigenvector with eigenvalue $\lambda$ it holds $\langle u,1\rangle=\langle T(u),1\rangle=\lambda\langle u,1\rangle$ due to \cref{ass:unbiased}.
Since $u$ has non-negative values (and is not identically zero) we know $\langle u,1 \rangle >0$. 
Hence, we can cancel $\langle u,1\rangle$ and obtain $\lambda=1$.
\end{proof}
Another issue is the invariance to a constant shift in illumination. We expect the behavior of $T$ to be invariant to a small global shift in image values. That is, $T(u+c) = T(u) + c$, for any $c \in \mathbb{R}$, such that $(u+c) \in U$.
\revision{This even applies to the proximal operators (analyzed in \Cref{Sec::proximal_operators_analysis}) associated to functionals, the null-space of which consists of constant functions, such as the total variation. 
In this case the eigenvalue problem $\lambda u = \prox_{\alpha(u)}^J(u)$ has only the trivial non-negative solution $u\equiv const$ with $\lambda=1$ and all other eigenfunctions have positive and negative entries. 
However, thanks to their null-space invariance \eqref{eq:shift_invariance_prox} the proximal eigenvalue problem can be relaxed to
\begin{align}
    \lambda(u-\overline{u}) = \prox_{\alpha}^J(u)-\overline{\prox_{\alpha}^J(u)}
\end{align}
which also allows for positive solutions.
Hence, we also relax the general basic} eigenproblem \cref{eq:EV} as follows:
\begin{align}\label{Eq:new_eigenproblem}
  \lambda ({u-\overline{u}}) = {T(u)-\overline{T(u)}},
\end{align}  
where $\lambda\in\mathbb{R}$, $\bar u=\langle 1,u \rangle /|\Omega|$ is the mean value of $u$ over the image domain $\Omega$. 
Note that now (relaxed) eigenvectors, admitting  \cref{Eq:new_eigenproblem}, can have any eigenvalue. 
In addition, if $u$ is an eigenvector, so is $u+c$ for $c\in\R$, as expected for operators with invariance to global value shifts.

Associated to this adapted eigenproblem, we define a suitable Rayleigh quotient as
\begin{align}
\label{eq:Rayleigh_new}
    R^\dagger(u)=\frac{\langle u-\overline{u},T(u)-\overline{T(u)}\rangle}{\norm{u-\overline{u}}^2}
\end{align}
which still has the property that $\lambda=R^\dagger(u)$ whenever $u$ fulfills the eigenvalue problem \cref{Eq:new_eigenproblem}.

\begin{rem}
A similar relaxation of the eigenvalue problem can be done for general Hilbert spaces if one replaces the mean $\overline{u}$ by the projection onto a closed subspace and $u-\overline{u}$ by the projection onto the respective orthogonal complement.
By choosing different subspaces (e.g., constant functions, affine functions, etc.) one can create different eigenvalue problems.
\end{rem}

\begin{algorithm}[h]
\caption{Nonlinear Power Method for Non-Homogeneous Operators.}
\label{alg:nonhom_power_it}
~~ \\
\textbf{Input:} $f\in \R^n$, $\varepsilon\in \R^+$.
\begin{enumerate}
    \item Initialize: $u^0 \gets f$, $\,\,k \gets 1.$ 
    \item Repeat until $\|u^{k+1}-u^k\|< \varepsilon$:\\ 
    $u^{k+1} \gets T(u^k)$\\
    $u^{k+1} \gets u^{k+1}-\overline{u^{k+1}}$\\
    $u^{k+1} \gets \frac{ u^{k+1}}{\| u^{k+1}\|}\|u^0-\overline{u^0} \|$\\
    $u^{k+1} \gets u^{k+1} + \overline{u^k}, \,\,\,k \gets k+1.$
\end{enumerate}
\textbf{Output:} $(u^*, \lambda^*)$, where $u^*=u^k$, $\lambda^*=R^\dagger(u^*)$, with $R^\dagger$ defined in \cref{eq:Rayleigh_new}.
\end{algorithm}

The modified power method is detailed in \cref{alg:nonhom_power_it}, aiming at computing a relaxed eigenvector \cref{Eq:new_eigenproblem} by explicitly handling the mean value and keeping the norm of the initial condition.
We found this adaptation to perform well on denoising networks.
It was shown that the eigenvectors obtained were of high eigenvalues ($\lambda \approx 1$), yielding highly smooth and stable structures. As proposed in \cite{HAITFRAENKEL2021103041}, one could also use a complementary operator $T^c := I-T$, where $I$ is the identity. In this case the proposed algorithm converged to unstable modes of the denoiser.


\section{Numerical results}
\revise{%
In this section we present numerical results for applying the power method to nonlinear operators of increasing complexity. We start with proximal operators in \Cref{Sec::proximal_operators} and proceed with shallow denoising networks in \Cref{Sec::shallow_networks}. 
In both of these scenarios we use the simple power method from \cref{alg:simple}.
This is justified by the range-independence of proximal operators and the shallow nets we use.
Finally, in \Cref{Sec::approx_eig} we apply the more general \cref{alg:nonhom_power_it} to state-of-the-art deep denoising networks in order to compute their eigenvectors in the sense of~\cref{Eq:new_eigenproblem}.
}
\revision{The code used used for our numerical experiments is available on \texttt{github}.\footnote{\url{https://github.com/leon-bungert/Eigenvectors-of-Proximal-Operators-and-Neural-Networks}}}

\subsection{\revise{Eigenvectors of Proximal Operators}}
\label{Sec::proximal_operators}
In this section we report numerical results for the proximal power method~\cref{eq:proximal_power_it}, which was analyzed in \Cref{Sec::proximal_operators_analysis}.
\revision{%
In our first experiment we compute an eigenfunction of the proximal operator of the anisotropic total variation functional
\begin{align*}
    \tv_\mathrm{aniso}(u):=\sup\left\lbrace
    \int_\Omega u\div\phi \dx\st \phi\in C_c^\infty(\Omega,\R^n),\;\sup_{x\in\Omega}|\phi(x)|_\infty\leq 1\right\rbrace,
\end{align*}
where $|\cdot|_\infty$ denotes the $\infty$-norm of a vector.
We compare our proximal power method \eqref{eq:proximal_power_it} with different parameter rules for $c=0.9$
\begin{alignat*}{2}
    \alpha(u) &= \frac{c\tau}{1+\tau J(u)},\quad&&\text{Feld $\tau$}, \\
    \alpha(u) &= \frac{c}{J(u^0)},\quad &&\text{constant}, \\
    \alpha(u) &= \frac{c}{J(u)},\quad &&\text{variable},
\end{alignat*}
with the flow from Nossek and Gilboa~\cite{nossek2018flows}, for which we choose the largest time step which guarantees stability.
Note that, as explained in \cref{ex:raleigh_feld}, using the parameter rule \mbox{Feld $\tau$} with time step size $\tau>0$ our proximal power method reduces to the Rayleigh quotient minimizing flow studied by Feld et al. in \cite{feld2019rayleigh}.
For all methods we perform 30 iterations.

\begin{wrapfigure}{r}{0.35\textwidth}
    \centering
    \includegraphics[width=0.17\textwidth,trim=.3cm .3cm .3cm .3cm, clip]{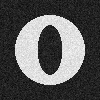}%
    \hfill%
    \includegraphics[width=0.17\textwidth,trim=.3cm .3cm .3cm .3cm, clip]{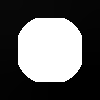}
    \caption{Initialization and converged eigenvector of the anisotropic total variation.}
    \label{fig:init_conv}
\end{wrapfigure}

The proximal operator of $\tv_\mathrm{aniso}$ cannot be evaluated explicitly and we approximate it using a primal-dual algorithm \cite{ChambollePock_PrimalDual}.
For parameters $\alpha(u)\approx J_*(u)$ close to the dual norm, the proximal operator is close to zero and particularly hard to evaluate accurately. 
Therefore, we employ a proximal backtracking strategy and evaluate the proximal operator with sufficient precision to ensure the energy decrease from \cref{cor:decrease_J}. 
If the energy has not decreased, we multiply the step size $\alpha(u)$ with $0.9$ and evaluate the proximal again. 
Alternatively, one could simply stop the iteration if the energy does not decrease any more.

All methods were initialized with a noisy image of a disk with a hole and we expect convergence to the anisotropic Cheeger set of the disk \cite{caselles2008characterization} which is an eigenfunction of the anisotropic total variation (cf.~\cref{fig:init_conv}).
Note that the curvature of the converged eigenvector is an artefact of the finite difference discretization (see, e.g., \cite{chambolle2021approximating}) and the displayed image is an eigenvector with very high numerical precision (cf.~\cref{fig:angles_affinities}).

In \cref{fig:angles_affinities} we plot the iterates' angles in degrees and eigenvector affinities (cf.~\cite{nossek2018flows,bungert2019computing}), which are defined as
\begin{align*}
    \text{Angle}(u) &:= \frac{360}{2\pi}\arccos\frac{\langle T(u),u\rangle}{\norm{T(u)}\norm{u}}, \\
    \text{Affinity}(u) &:= \frac{\norm{p}^2}{J(p)}\in[0,1],\quad p\in\partial J(u).
\end{align*}
For the proximal power methods we choose $T(u)=\prox_{\alpha(u)}^J(u)$ and for the Nossek and Gilboa flow $T(u)=\partial J(u)$.
The appropriate subgradient needed for the computation of the eigenvector affinity of the proximal power method is given by $p:=(u-T(u))/\alpha(u)$.

By \cref{lem:collinearity} the angle converges to zero for the proximal power method, whereas the eigenvector affinity is expected to converge to one.
\cref{fig:angles_affinities} makes clear that the proximal power method~\eqref{eq:proximal_power_it} with variable parameter rule converges fastest, followed by the Feld et al. method with large time step $\tau=1$, which approximates the variable parameter rule according to \cref{ex:raleigh_feld}.
These two methods converge to high precision eigenvectors after four and five iterations, respectively.
Feld et al.'s method with time step $\tau=0.1$ and the power method with constant parameter rule converge slower and Feld et al.'s method with small time step $\tau=0.01$ and the Nossek and Gilboa flow converge slowest, without reaching an eigenvector in 30 iterations.
This is due to the fact that the noisy initialization requires quite small time steps for stability whereas the power method with variable parameter rule can increase its parameter in every iteration.

Interestingly, the iterates of the ``slower methods'' exhibit smaller angles than the ``faster'' ones.
Similarly, the eigenvector affinities, which all lie above $99\%$ are slightly higher for the slow methods.
The method Feld $0.1$, which needs more than ten iterations to converge, achieves the highest affinity of $\approx 99.9\%$, whereas the variable parameter rule, which converges in only 4 iterations, reaches an affinity of $\approx 99.5\%$.
We attribute these observations to the fact that the larger the parameter of the proximal operator, the harder it is to evaluate it  precisely which can lead to slightly larger angles and lower affinities. 

The iterates of all six methods are visualized in \cref{fig:comparison_images}, from which it is obvious that the proposed method (and the special case of Feld et al. \cite{feld2019rayleigh} with large $\tau$) converges much faster than the flow from Nossek and Gilboa~\cite{nossek2018flows}.}
\begin{figure}[h!]
    \centering
    \def\PicWidth{0.49\textwidth}
    \includegraphics[width=\PicWidth,trim=.5cm .4cm .4cm .4cm,clip]{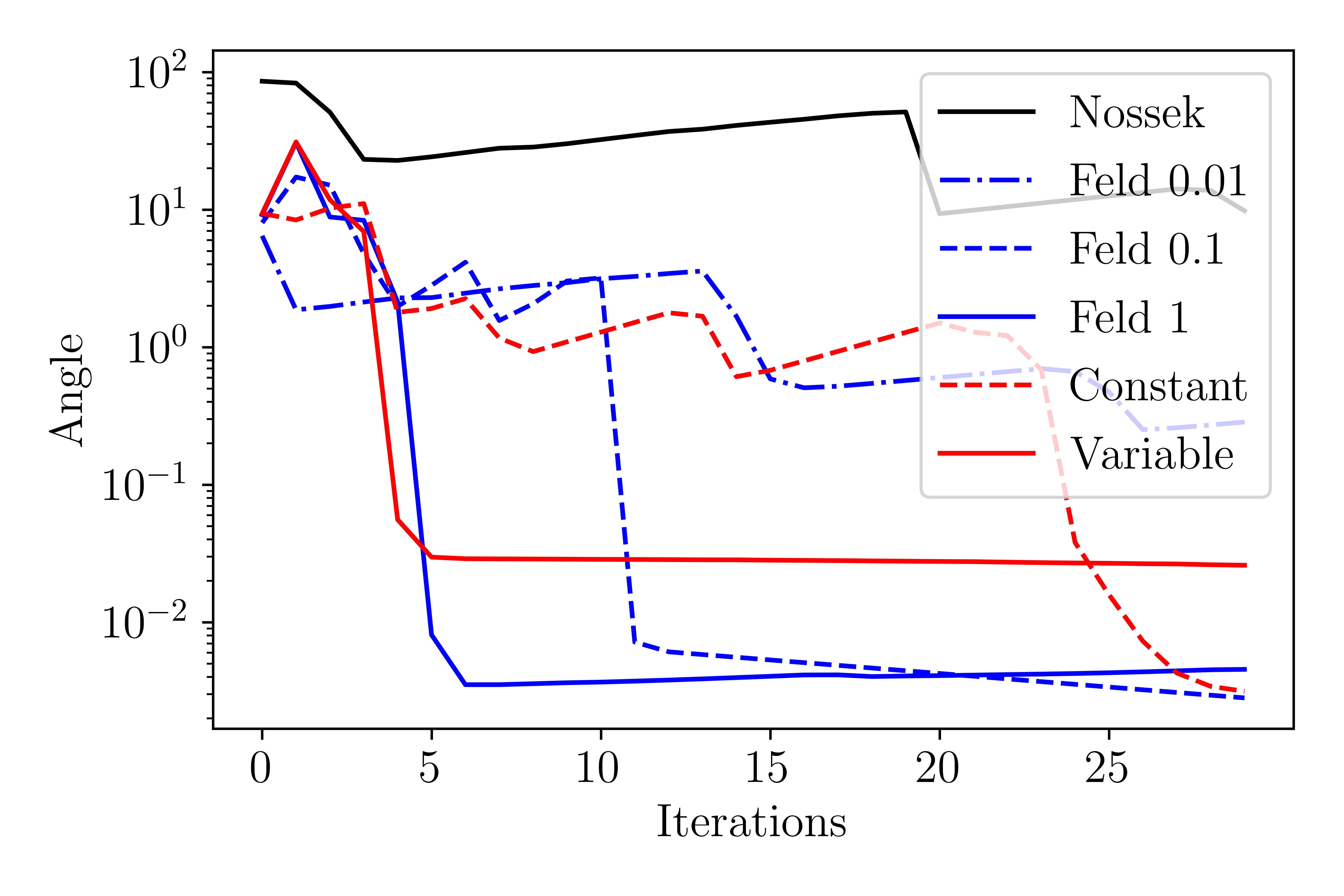}%
    \hfill%
    \includegraphics[width=\PicWidth,trim=.5cm .4cm .4cm .4cm,clip]{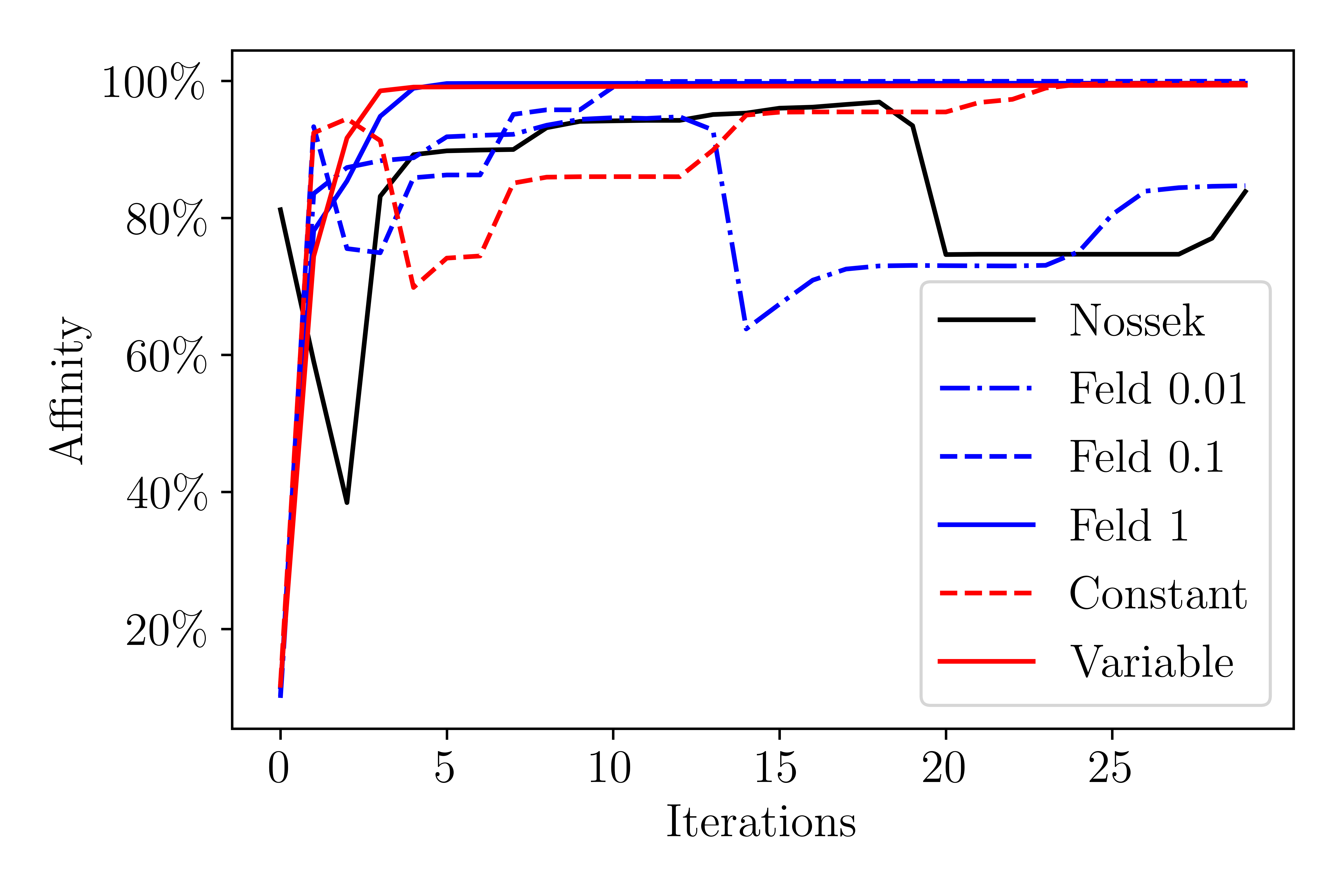}
    \caption{Angles (\textbf{left}) and eigenvector affinities (\textbf{right}) for different methods (\textbf{red}: proposed). The proximal power method \eqref{eq:proximal_power_it} with variable parameter rule including the Feld et al. method \cite{feld2019rayleigh} for large time steps converge fastest.}
    \label{fig:angles_affinities}
\end{figure}
\begin{figure}[h!]
    \setlength{\tabcolsep}{2pt}
    \gdef\Width{0.17\textwidth}
    \centering
    \begin{tabular}{cccccc}
    \rotatebox{90}{\hspace{.2cm}\begin{tabular}[c]{@{}c@{}}Nossek~\cite{nossek2018flows}\\maximal $\tau$\end{tabular}} 
    &\includegraphics[width=\Width,trim=.3cm .3cm .3cm .3cm, clip]{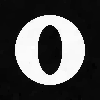}%
    &\includegraphics[width=\Width,trim=.3cm .3cm .3cm .3cm, clip]{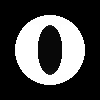}%
    &\includegraphics[width=\Width,trim=.3cm .3cm .3cm .3cm, clip]{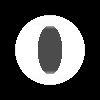}%
    &\includegraphics[width=\Width,trim=.3cm .3cm .3cm .3cm, clip]{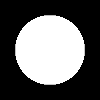}%
    &\includegraphics[width=\Width,trim=.3cm .3cm .3cm .3cm, clip]{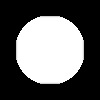}\\%
    \rotatebox{90}{\hspace{.5cm}\begin{tabular}[c]{@{}c@{}}Feld~\cite{feld2019rayleigh}\\$\tau=0.01$\end{tabular}} 
    &\includegraphics[width=\Width,trim=.3cm .3cm .3cm .3cm, clip]{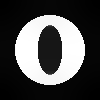}%
    &\includegraphics[width=\Width,trim=.3cm .3cm .3cm .3cm, clip]{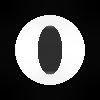}%
    &\includegraphics[width=\Width,trim=.3cm .3cm .3cm .3cm, clip]{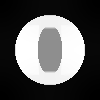}%
    &\includegraphics[width=\Width,trim=.3cm .3cm .3cm .3cm, clip]{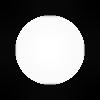}%
    &\includegraphics[width=\Width,trim=.3cm .3cm .3cm .3cm, clip]{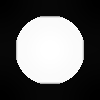}\\%
    \rotatebox{90}{\hspace{.5cm}\begin{tabular}[c]{@{}c@{}}Feld~\cite{feld2019rayleigh}\\$\tau=0.1$\end{tabular}} 
    &\includegraphics[width=\Width,trim=.3cm .3cm .3cm .3cm, clip]{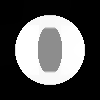}%
    &\includegraphics[width=\Width,trim=.3cm .3cm .3cm .3cm, clip]{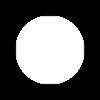}%
    &\includegraphics[width=\Width,trim=.3cm .3cm .3cm .3cm, clip]{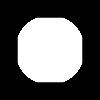}%
    &\includegraphics[width=\Width,trim=.3cm .3cm .3cm .3cm, clip]{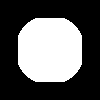}%
    &\includegraphics[width=\Width,trim=.3cm .3cm .3cm .3cm, clip]{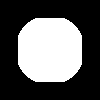}\\%
    \rotatebox{90}{\hspace{.5cm}\begin{tabular}[c]{@{}c@{}}Feld~\cite{feld2019rayleigh}\\$\tau=1$\end{tabular}} 
    &\includegraphics[width=\Width,trim=.3cm .3cm .3cm .3cm, clip]{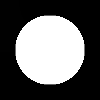}%
    &\includegraphics[width=\Width,trim=.3cm .3cm .3cm .3cm, clip]{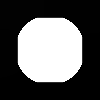}%
    &\includegraphics[width=\Width,trim=.3cm .3cm .3cm .3cm, clip]{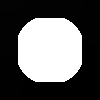}%
    &\includegraphics[width=\Width,trim=.3cm .3cm .3cm .3cm, clip]{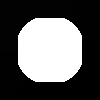}%
    &\includegraphics[width=\Width,trim=.3cm .3cm .3cm .3cm, clip]{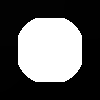}
    \\%
    \rotatebox{90}{\hspace{.2cm}\begin{tabular}[c]{@{}c@{}}Ours~\eqref{eq:proximal_power_it}\\constant rule\end{tabular}} 
    &\includegraphics[width=\Width,trim=.3cm .3cm .3cm .3cm, clip]{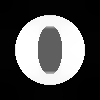}%
    &\includegraphics[width=\Width,trim=.3cm .3cm .3cm .3cm, clip]{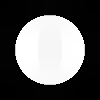}%
    &\includegraphics[width=\Width,trim=.3cm .3cm .3cm .3cm, clip]{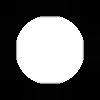}%
    &\includegraphics[width=\Width,trim=.3cm .3cm .3cm .3cm, clip]{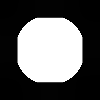}%
    &\includegraphics[width=\Width,trim=.3cm .3cm .3cm .3cm, clip]{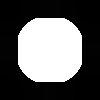}
    \\%
    \rotatebox{90}{\hspace{.2cm}\begin{tabular}[c]{@{}c@{}}Ours~\eqref{eq:proximal_power_it}\\variable rule\end{tabular}} 
    &\includegraphics[width=\Width,trim=.3cm .3cm .3cm .3cm, clip]{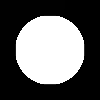}%
    &\includegraphics[width=\Width,trim=.3cm .3cm .3cm .3cm, clip]{figs/proximal_operators/variable/it_4}%
    &\includegraphics[width=\Width,trim=.3cm .3cm .3cm .3cm, clip]{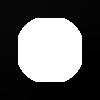}%
    &\includegraphics[width=\Width,trim=.3cm .3cm .3cm .3cm, clip]{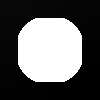}%
    &\includegraphics[width=\Width,trim=.3cm .3cm .3cm .3cm, clip]{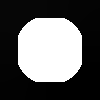}%
    \end{tabular}
    \caption{Iterates 2, 4, 10, 20, and 30 of different methods to compute eigenvectors. \textbf{Top to bottom:} Nossek and Gilboa~\cite{nossek2018flows} with maximal time step, Feld et al. \cite{feld2019rayleigh} with time step $0.01$, $0.1$, and $1$, proximal power method \eqref{eq:proximal_power_it} with constant and variable parameter rule.}
    \label{fig:comparison_images}
\end{figure}
While these results were computed on a finite difference grid, it is also possible to compute nonlinear eigenfunctions on nonlocal weighted graphs.
The three images in \cref{Fig::proximal} correspond to eigenfunctions related to the absolutely one-homogeneous functionals
\begin{align}
    J_p(u) = 
    \begin{cases}
    \left(\sum_{x\in V}\sum_{y\sim x}|\nabla_w u(x,y)|^p\right)^\frac{1}{p},\quad &p\in[1,\infty), \\
    \max_{x\in V}\max_{y\sim x}|\nabla_w u(x,y)|,\quad &p=\infty,
    \end{cases}
\end{align}
which are defined on a weighted graph $G=(V,E,w)$.
Here $V$ denotes the set of vertices, $E\subset V\times V$ is the set of edges, and $w:E\to\R_{\geq 0}$ is a symmetric weight function.
For $x\sim y$, meaning $(x,y)\in E$, one defines the gradient of a function $u:V\to\R$ as
\begin{align*}
    \nabla_w u(x,y) =\sqrt{w(x,y)}(u(x)-u(y)).
\end{align*}
For $p=1$ the eigenfunctions of the proximal operator of $J_1$ can be used for the task of graph clustering (cf.~\cite{aujol2018theoretical,bungert2019computing}). 
The first image in \cref{Fig::proximal} shows the corresponding eigenfunction for the clustering of a graph consisting of two ``moons''.
Similarly, for $p=2$ one obtains the standard graph Laplacian eigenfunction (see the second image in \cref{Fig::proximal}).
Finally, one can also consider the functionals $J_p$ on the space of graph functions which are zero on some set $\Gamma\subset V$.
In this case, for $p=\infty$ positive eigenfunctions related to $J_\infty$ were characterized as distance function to the set $\Gamma$ in~\cite{bungert2020structural}.
The third image in \cref{Fig::proximal} shows the distance function to the boundary of a map of the United Kingdom, computed with the proximal power method.
\revision{All results in \cref{Fig::proximal} were computed using the variable parameter rule with $c=0.9$.}
\begin{figure}[h!]
    \centering
    \includegraphics[width=0.39\textwidth,angle=90,trim=4cm 4cm 4cm 3.5cm, clip, angle=90]{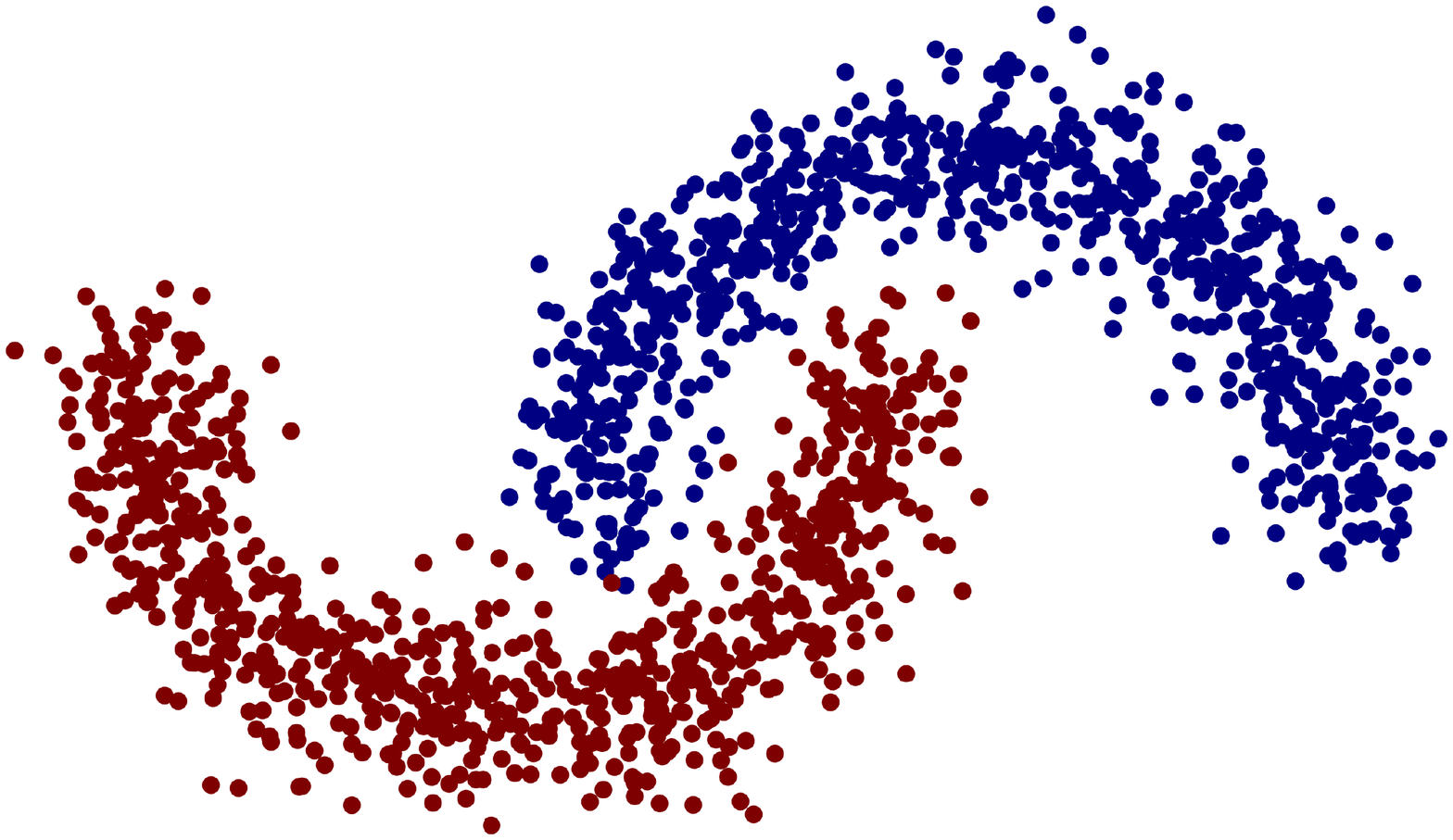}%
    \hfill%
    \includegraphics[width=0.39\textwidth,angle=90,trim=6.5cm 4cm 5.5cm 3.5cm , clip, angle=90]{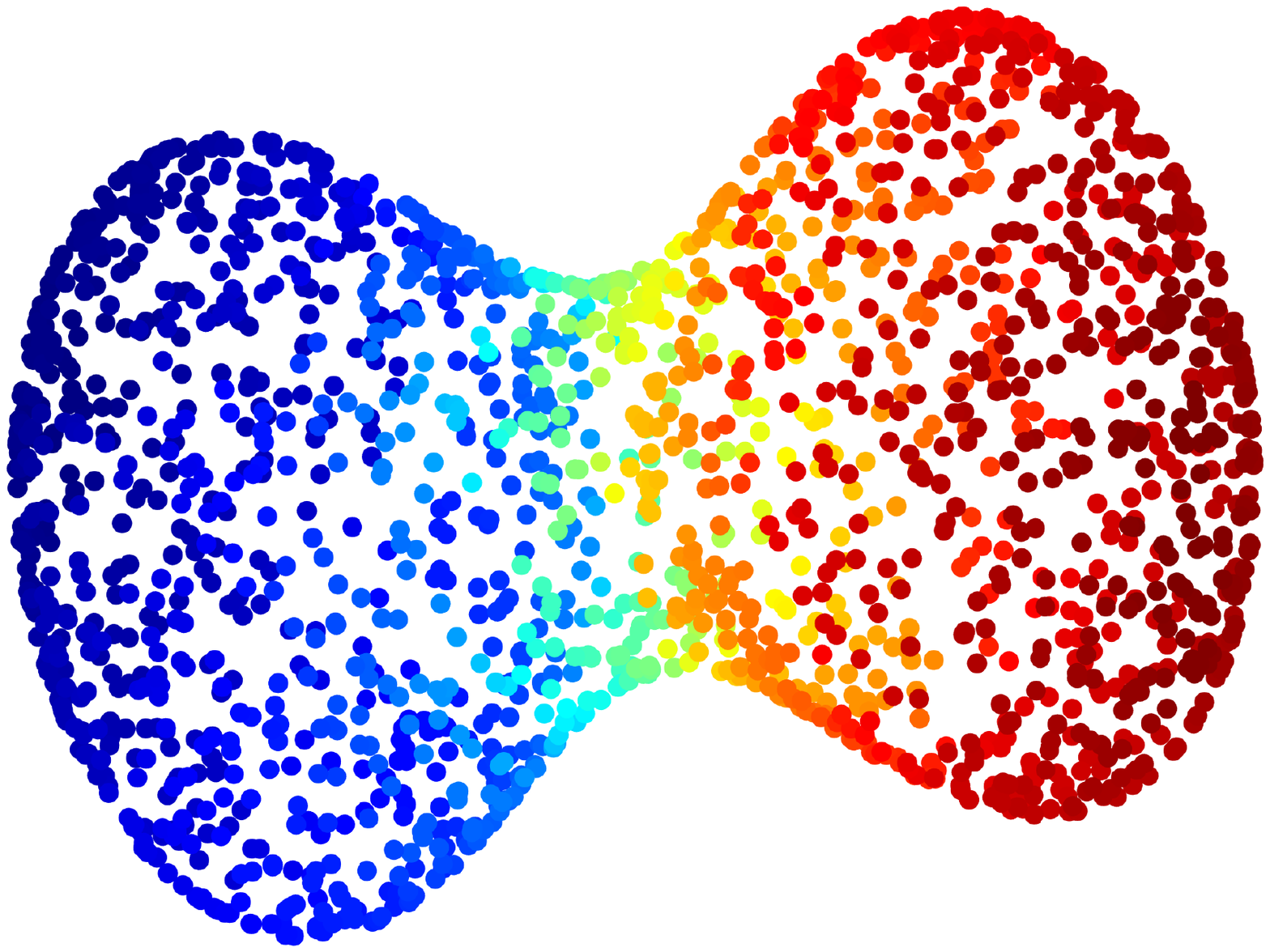}%
    \hfill%
    \includegraphics[width=0.2\textwidth, trim=3cm 3cm 2cm 2cm, clip]{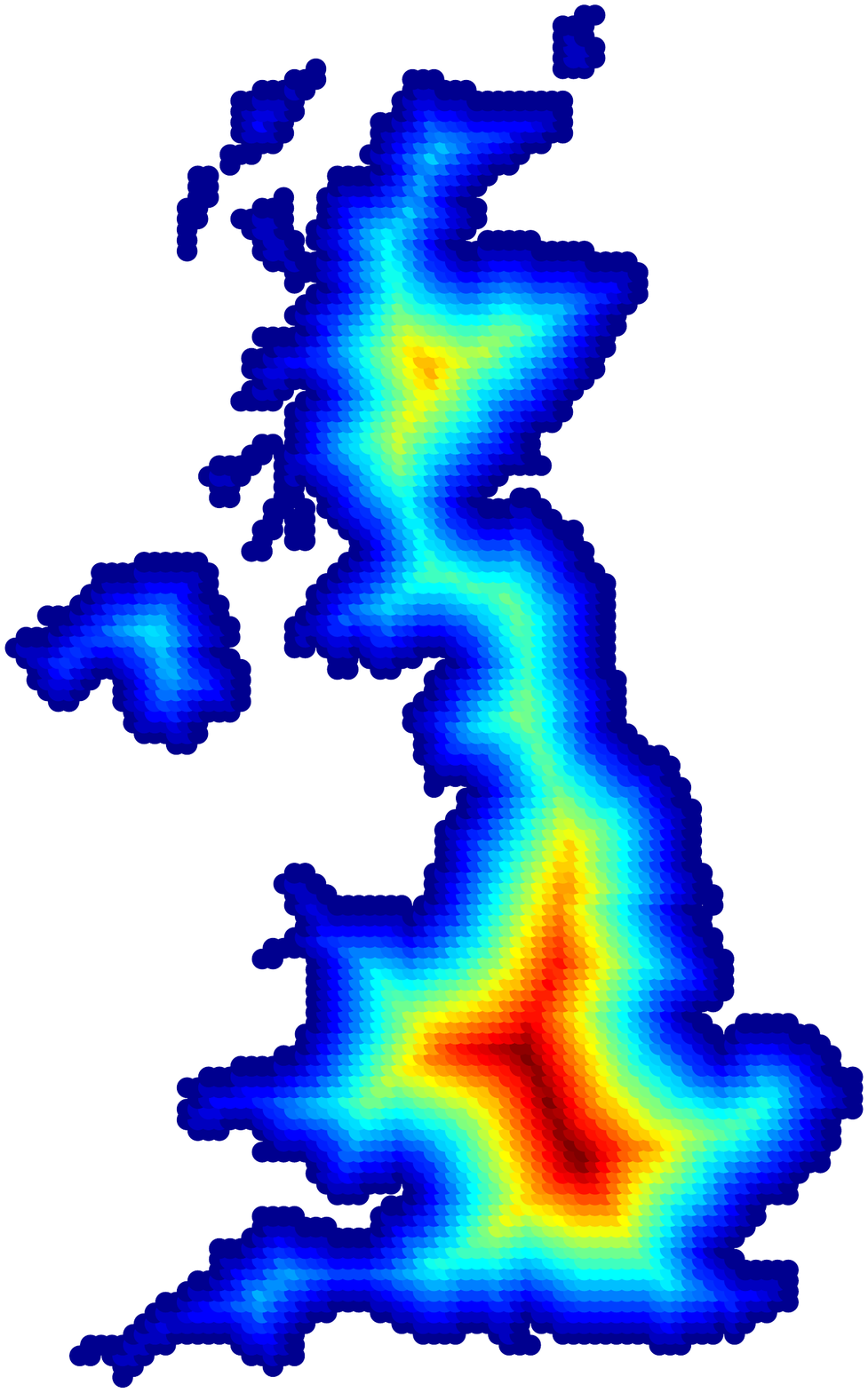}
    \caption{Eigenfunctions of different proximal operators:
    eigenfunction related to $J_1$, graph Laplacian eigenfunction related to $J_2$, distance function related to $J_\infty$}
    \label{Fig::proximal}
\end{figure}

\subsection{Eigenvectors of Shallow Denoising Networks}
\label{Sec::shallow_networks}
In this section, we present experimental results for the following basic denoising networks: 
\begin{align*}
    T(u)=
    \begin{cases}
    Au&\texttt{Linear},\\
    \sigma(Au)&\texttt{ReLU},\\
    \sigma(A_2\sigma(A_1u))&\texttt{AutoEncoder},\\
    \sigma(B\sigma(Cu))&\texttt{ConvNet},
    \end{cases}
    \quad u\in\R^n
\end{align*}
where $\sigma$ denotes $\relu$ activation and $n=28^2=784$ denotes image size.
The involved matrices in the \texttt{Linear}, \texttt{ReLU}, and \texttt{AutoEncoder} nets are $A\in\R^{n\times n}$, $A_1\in\R^{f\times n}$, and $A_2\in\R^{n\times f}$, where  \texttt{SmallAE} and \texttt{LargeAE} denote feature space sizes $f=25$ and $f=200$, respectively. 
The convolutional layer $C$ in \texttt{ConvNet} produces 16 kernels of size $6\times 4$, and the subsequent fully connected layer corresponds to a matrix $B\in\R^{n\times n}$.
The networks were trained on 5000 pairs of noisy and clean images from the MNIST database with values in $[0,1]$.

We use the simple \cref{alg:simple} to compute eigenvectors of the net. 
This seems reasonable, since all nets are positively homogeneous due to the lack of bias terms.
All eigenvectors computed fulfill $\norm{\lambda u-T(u)}<10^{-7}$ where $\lambda=\norm{T(u)}$ denotes the eigenvalue.
This means they fulfill the eigenvector relation with high numerical accuracy.
The eigenvectors and corresponding eigenvalues are shown in \cref{fig:mnist_eigenvectors}.
Furthermore, \cref{tab:rmse} shows the Root Mean Squared Errors (RMSE) of the trained networks on an unseen test set of MNIST digits.%
\begin{figure*}[h!]
\captionsetup[subfigure]{justification=centering}
\centering
\begin{subfigure}{0.155\textwidth}
\centering
    \includegraphics[width=0.9\textwidth]{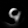}
    \caption{\texttt{Linear} $\lambda=0.9626$}
\end{subfigure}
\begin{subfigure}{0.155\textwidth}
\centering
    \includegraphics[width=0.9\textwidth]{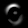}
    \caption{\texttt{Linear} $\lambda=0.9672$}
\end{subfigure}
\begin{subfigure}{0.155\textwidth}
\centering
    \includegraphics[width=0.9\textwidth]{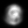}
    \caption{\texttt{Linear} $\lambda=0.9996$}
\end{subfigure}
\begin{subfigure}{0.155\textwidth}
\centering
    \includegraphics[width=0.9\textwidth]{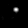}
    \caption{\texttt{ReLU} $\lambda=1.0443$}
\end{subfigure}
\begin{subfigure}{0.155\textwidth}
\centering
    \includegraphics[width=0.9\textwidth]{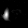}
    \caption{\texttt{ReLU} $\lambda=1.1808$}
\end{subfigure}
\begin{subfigure}{0.155\textwidth}
\centering
    \includegraphics[width=0.9\textwidth]{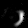}
    \caption{\texttt{ReLU} $\lambda=1.2165$}
\end{subfigure}
\begin{subfigure}{0.155\textwidth}
\centering
    \includegraphics[width=0.9\textwidth]{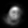}
    \caption{\texttt{SmallAE} $\lambda=0.9960$}
\end{subfigure}
\begin{subfigure}{0.155\textwidth}
\centering
    \includegraphics[width=0.9\textwidth]{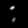}
    \caption{\texttt{LargeAE} $\lambda=1.0310$}
\end{subfigure}
\begin{subfigure}{0.155\textwidth}
\centering
    \includegraphics[width=0.9\textwidth]{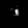}
    \caption{\texttt{LargeAE} $\lambda=1.0670$}
\end{subfigure}
\begin{subfigure}{0.155\textwidth}
\centering
    \includegraphics[width=0.9\textwidth]{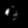}
    \caption{\texttt{LargeAE} $\lambda=1.0961$}
\end{subfigure}
\begin{subfigure}{0.155\textwidth}
\centering
    \includegraphics[width=0.9\textwidth]{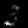}
    \caption{\texttt{LargeAE} $\lambda=1.1330$}
\end{subfigure}
\begin{subfigure}{0.155\textwidth}
\centering
    \includegraphics[width=0.9\textwidth]{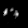}
    \caption{\texttt{LargeAE} $\lambda=1.1358$}
\end{subfigure}
\begin{subfigure}{0.155\textwidth}
\centering
    \includegraphics[width=0.9\textwidth]{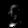}
    \caption{\texttt{LargeAE} $\lambda=1.1393$}
\end{subfigure}
\begin{subfigure}{0.155\textwidth}
\centering
    \includegraphics[width=0.9\textwidth]{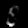}
    \caption{\texttt{LargeAE} $\lambda=1.1617$}
\end{subfigure}
\begin{subfigure}{0.155\textwidth}
\centering
    \includegraphics[width=0.9\textwidth]{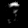}
    \caption{\texttt{LargeAE} $\lambda=1.1801$}
\end{subfigure}
\begin{subfigure}{0.155\textwidth}
\centering
    \includegraphics[width=0.9\textwidth]{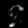}
    \caption{\texttt{LargeAE} $\lambda=1.1963$}
\end{subfigure}
\begin{subfigure}{0.155\textwidth}
\centering
    \includegraphics[width=0.9\textwidth]{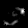}
    \caption{\texttt{LargeAE} $\lambda=1.2174$}
\end{subfigure}
\begin{subfigure}{0.155\textwidth}
\centering
    \includegraphics[width=0.9\textwidth]{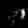}
    \caption{\texttt{LargeAE} $\lambda=1.2687$}
\end{subfigure}
\begin{subfigure}{0.155\textwidth}
\centering
    \includegraphics[width=0.9\textwidth]{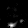}
    \caption{\texttt{ConvNet} $\lambda=1.3319$}
\end{subfigure}
\begin{subfigure}{0.155\textwidth}
\centering
    \includegraphics[width=0.9\textwidth]{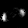}
    \caption{\texttt{ConvNet} $\lambda=1.4390$}
\end{subfigure}
\begin{subfigure}{0.155\textwidth}
\centering
    \includegraphics[width=0.9\textwidth]{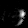}
    \caption{\texttt{ConvNet} $\lambda=1.4523$}
\end{subfigure}
\vspace{-15pt}
\caption{Eigenvectors and eigenvalues for the different shallow denoising networks\label{fig:mnist_eigenvectors}. 
\texttt{Linear} eigenvectors are computed with standard methods, all others using \cref{alg:simple} with different initializations.}
\vspace{-25pt}
\end{figure*}
\begin{table}[h!]
    \centering
    \begin{tabular}{c|c|c|c|c|c}
          Network & \texttt{Linear} & \texttt{ReLU} & \texttt{SmallAE} & \texttt{LargeAE} & \texttt{ConvNet}\\
          \hline
          RMSE & 0.097345 & 0.083013 & 0.13968 & 0.097402 & 0.083913 
    \end{tabular}
    \caption{Root Mean Squared Errors of the networks on an unseen test set of MNIST digits}
    \label{tab:rmse}
\end{table}

One should remark that the eigenvectors of the \texttt{Linear} net were computed with standard methods\footnote{We used the MATLAB routine \texttt{eig}.}, since they coincide with eigenvectors of the matrix $A$.
Indeed, in our experiments $A$ had several eigenvectors with real eigenvalues, but we only show the three with the largest eigenvalue. 
Similarly, for the other nets we show all eigenvectors that we were able to compute using different initializations of \cref{alg:simple}.
For each network, we initialized the power method with 36 different images, being zero apart from a single pixel with varying position.

While by this method most nets only exhibit one to three different eigenvectors, the \texttt{LargeAE} net shows eleven different eigenvectors.
Furthermore, some of the eigenvectors of the \texttt{Linear} and \texttt{LargeAE} nets show remarkable similarity with handwritten digits. 
This could indicate overfitting, which is also supported by the relatively low RMSE of these nets (\cref{tab:rmse}).
On the other hand, the \texttt{SmallAE} network, having by far the smallest number of parameters, exhibits only one eigenvector 
and has larger RMSE.
This might indicate that this network is not very much bound to the training data and has few invariances.

\subsection{Approximate Eigenvectors as Stable and Unstable Modes of Deep Denoising Nets}
\label{Sec::approx_eig}

In this section, we analyze eigenproblems for deep denoising nets. Directly applying \cref{alg:simple} for these nets results in less meaningful results (cf.~\cref{Fig::simple_for_FFDnet}), due their unique nature, as detailed in \Cref{Sec::adapted}. Thus, we use the adapted eigenproblem \cref{Eq:new_eigenproblem} and \cref{alg:nonhom_power_it}. 

The resulting solutions can only be regraded as approximate eigenvectors, since they do not accurately solve the eigenproblem (see \cref{def::approx_eigen}, using degrees throughout the paper). Thus, we address these solutions as stable and unstable modes of the denoiser. Stable modes are optimal inputs for the denoiser, achieving superior PSNR in noise removal, and corresponding to large eigenvalues. We also show a method for generating unstable modes, which are strongly suppressed by the denoiser and correspond to small eigenvalues. 

\begin{wrapfigure}{r}{0.6\textwidth}
    \centering
    \begin{subfigure}{0.19\textwidth}
    \centering
    \includegraphics[height=1.8cm]{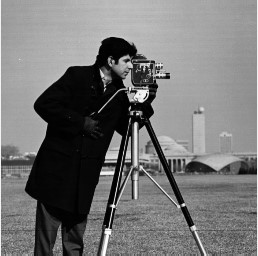}
    \caption{Initial\\ image}
\end{subfigure}
\begin{subfigure}{0.19\textwidth}
\centering
    {\includegraphics[height=1.8cm,trim=0pt 10pt 0pt 10pt,clip]{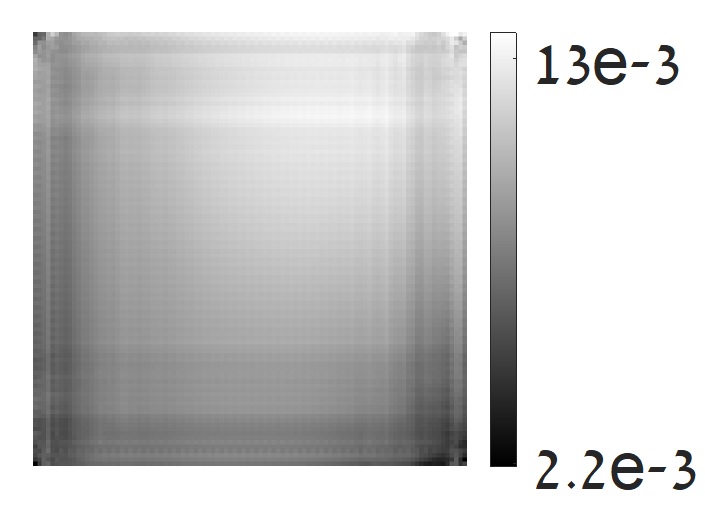}}
    \caption{Stable mode with unit norm}
\end{subfigure}
\begin{subfigure}{0.19\textwidth}
\centering
    {\includegraphics[height=1.8cm]{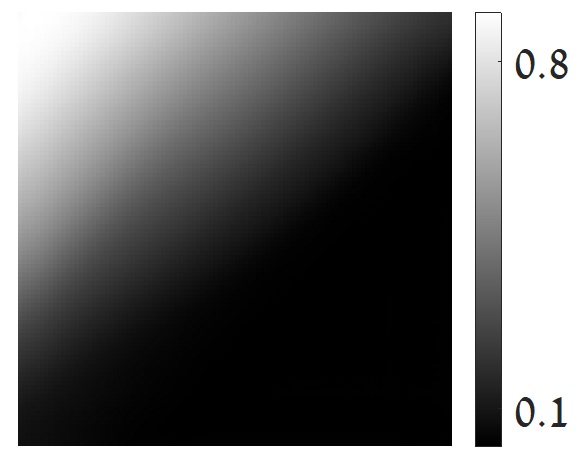}}
    \caption{Stable mode with original norm}
\end{subfigure}
\vspace{-15pt}
\caption{FFDnet's eigenvectors using the simple power method \cref{alg:simple} (initial condition \texttt{cameraman}). $\langle T(u),1 \rangle = \langle u,1 \rangle$ holds only approximately. (b) normalizing to unit norm: $\lambda=1.123$, $\theta_{2000}=0^\circ$. (c) normalizing initial norm: $\lambda=0.9999$, $\theta_{2000}=0.56^\circ$.}
\label{Fig::simple_for_FFDnet}
\vspace{-25pt}
\end{wrapfigure}

\cref{Fig::better_removal} shows a known behavior of stable and unstable modes: the denoiser better removes noise from its stable modes, than from natural images. Much poorer noise removal is achieved when denoising its unstable modes, which are approximately as unstable as noise itself.
We show such stable and unstable modes for the FFDnet~\cite{zhang2018ffdnet} (\Cref{Sec::FFDnet}) and DnCNN~\cite{zhang2017beyond} (\Cref{Sec::DnCNN}) deep denoising nets. Throughout our experiments, we validate these are stable modes following \cref{def::approx_eigen} and \cref{prop:Ray} and~\cref{prop:theta} (see~\cite{HAITFRAENKEL2021103041}). First, the Rayleigh quotient $R^\dagger(u)$ stabilizes to a value $\lambda$. Second, the angle $\theta$ between $u^k$ and $T(u^k)$ decreases to approximately zero. Also, the eigenproblem~\cref{Eq:new_eigenproblem} approximately holds for each pixel. 

\begin{figure*}[h]
\captionsetup[subfigure]{justification=centering}
\centering
\begin{subfigure}{0.17\textwidth}
\centering
    \includegraphics[width=1.8cm]{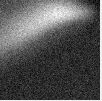}
    \caption{\\Noisy stable mode}
\end{subfigure}
\begin{subfigure}{0.15\textwidth}
\centering
    \includegraphics[width=1.8cm]{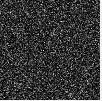}
    \caption{Noise removed from (a)}
\end{subfigure}
\begin{subfigure}{0.15\textwidth}
\centering
    \includegraphics[width=1.8cm]{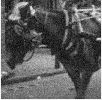}
    \caption{\\Noisy natural image}
\end{subfigure}
\begin{subfigure}{0.15\textwidth}
\centering
    \includegraphics[width=1.8cm]{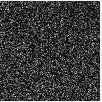}
    \caption{Noise removed from (c)}
\end{subfigure}
\begin{subfigure}{0.17\textwidth}
\centering
    \includegraphics[width=1.8cm]{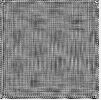}
    \caption{\\Noisy unstable mode}
\end{subfigure}
\begin{subfigure}{0.15\textwidth}
\centering
    \includegraphics[width=1.8cm]{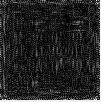}
    \caption{Noise removed from (e)}
\end{subfigure}
\\
\begin{subfigure}{\textwidth}
\centering
    \begin{tikzpicture}
    \begin{axis}[
        title={FFDnet Denoising - PSNR Gain},
        ybar,
        width=0.9\textwidth,
        height=6cm,
        bar width = 20pt,
        enlargelimits=0.2,
        legend cell align=left,
        legend pos =  north west,
        ylabel={PSNR Gain (dB)},
        ytick={-10,0,10,20},
        symbolic x coords={1,2,3,4,5,6,7,8,9,10,11,12},
        xmajorticks=false
        ]
        \addplot[fill=violet] coordinates {(1,-11.78)};
        \addplot[fill=red] coordinates {(2,-11.8544) (3,-11.7956) (4,-11.7759)};
        \addplot[fill=green] coordinates {(5,3.8982) (6,5.1825) (7,7.7979)};
        \addplot[fill=lime] coordinates {(8,12.3467) (9,14.8279) (10,19.1838) (11,19.3078) (12,20.8687)};
        \legend{noise,unstable modes,natural images,stable modes}
    \end{axis}
    \end{tikzpicture}
    \caption{PSNR gain when denoising different noisy images: stable modes vs. natural images vs. unstable modes vs. noise images. The noise variance of the noise is $1/15$ of the image variance.}
\end{subfigure}
\vspace{-15pt}
\caption{Illustration of a known property for stable and unstable modes. FFDnet achieves best denoising results for its stable modes (a-b), medium results for natural images (c-d), and worst results for unstable modes (e-f). The noise removed from the stable mode is also more uniform (note that (f) contains some structure of (e)). (g): PSNR gain when denoising using FFDnet for different noisy images (noise image result averaged over 50 images).}
\label{Fig::better_removal}
\vspace{-25pt}
\end{figure*}
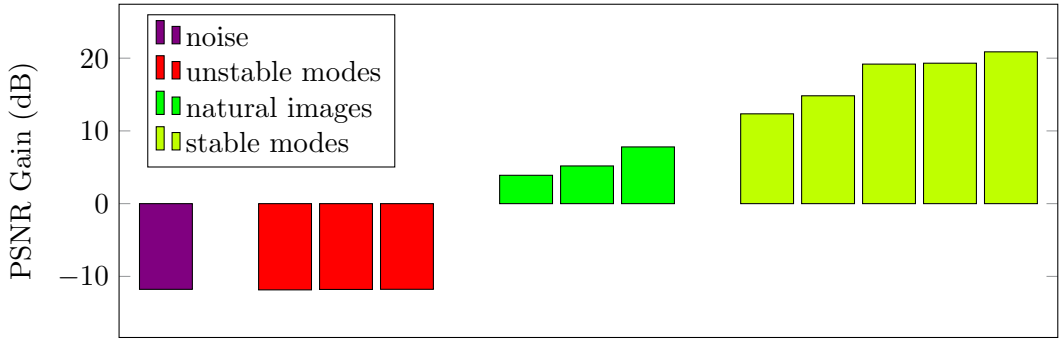

\subsubsection{FFDnet Deep Denoising Net}\label{Sec::FFDnet}
We first present results of \cref{alg:nonhom_power_it} for the FFDnet deep denoising~\cite{zhang2018ffdnet}. 

\paragraph{Nature of Operator} We demonstrate the smoothing nature of FFDnet by iterative denoising (no power method) using 500 iterations in \cref{Fig::FFDnet_nature}. 

\paragraph{Stable Modes} \cref{Fig::FFDnet} shows the evolution of the power method to the final stable modes, using FFDnet on three different initial images. Eigenvalues are larger than but very close to 1, which is unusual for a denoiser (see Cor.~1 in~\cite{HAITFRAENKEL2021103041}). The stable modes differ from each other, each depending on its initial condition. Modes were validated following the framework detailed in \Cref{Sec::approx_eig}. 

\begin{figure*}
\captionsetup[subfigure]{justification=centering}
\centering
\begin{subfigure}{0.155\textwidth}
\centering
    \includegraphics[width=0.9\textwidth]{cameraman.jpg}
    \caption{Initial image}
\end{subfigure}
\begin{subfigure}{0.155\textwidth}
\centering
    \includegraphics[width=0.9\textwidth]{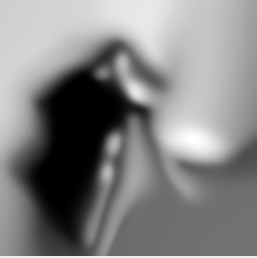}
   \caption{100 iterations}
\end{subfigure}
\begin{subfigure}{0.155\textwidth}
\centering
    \includegraphics[width=0.9\textwidth]{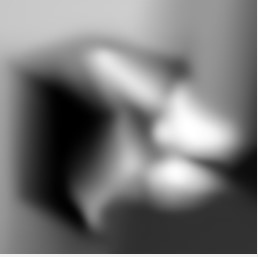}
    \caption{300 iterations}
\end{subfigure}
\begin{subfigure}{0.155\textwidth}
\centering
    \includegraphics[width=0.9\textwidth]{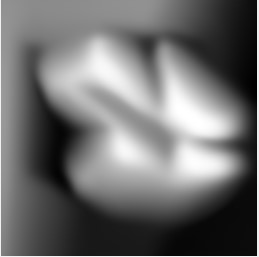}
    \caption{600 iterations}
\end{subfigure}
\begin{subfigure}{0.155\textwidth}
\centering
    \includegraphics[width=0.9\textwidth]{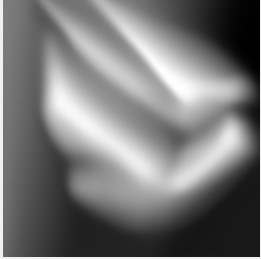}
    \caption{1000 iterations}
\end{subfigure}
\begin{subfigure}{0.155\textwidth}
\centering
    \includegraphics[width=0.9\textwidth]{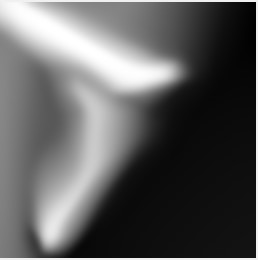}
    \caption{\\Stable mode}
\end{subfigure}
\\
\begin{subfigure}{0.24\textwidth}
\centering
    \includegraphics[height=2cm]{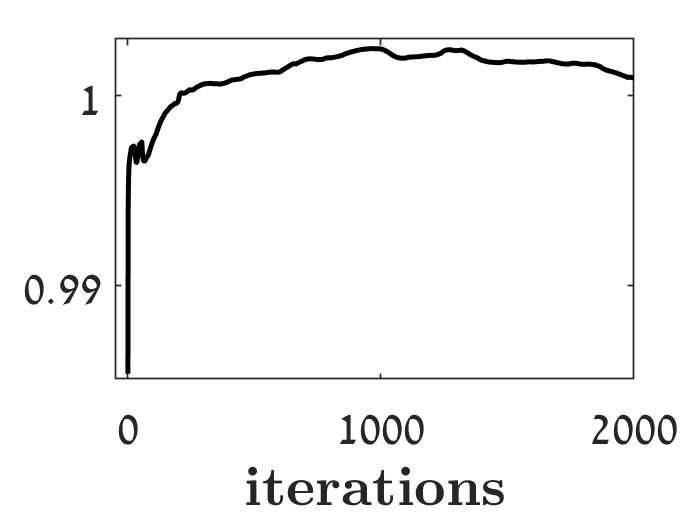}
    \caption{Rayleigh quotient $R^\dagger(u)$ $\rightarrow \lambda$}
\end{subfigure}
\begin{subfigure}{0.45\textwidth}
\centering
    \includegraphics[height=2cm]{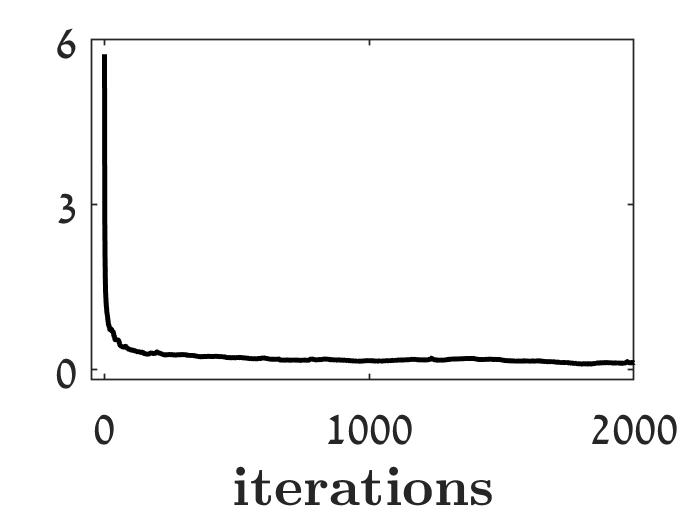}
    \caption{$\theta$ between $u^k, T(u^k) \rightarrow 0$,\\$\theta_{2000}=0.11^\circ$}
\end{subfigure}
\begin{subfigure}{0.24\textwidth}
\centering
    \includegraphics[height=2cm]{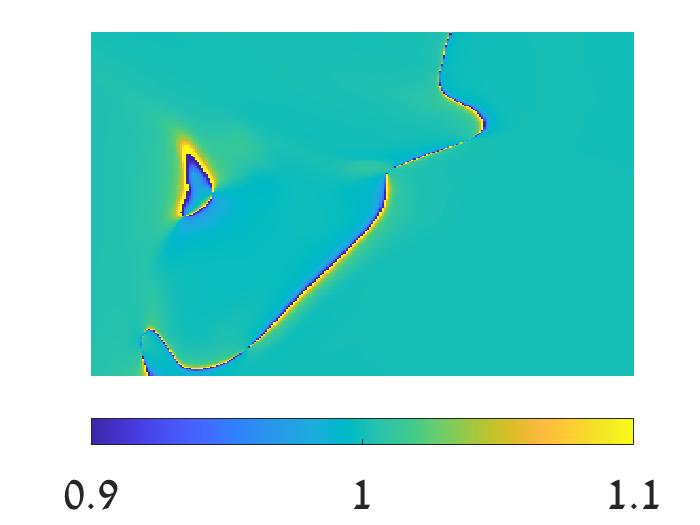}
    \caption{$\lambda$ per pixel: \cref{Eq:new_eigenproblem} holds approximately}
\end{subfigure}
\\
\begin{subfigure}{0.155\textwidth}
\centering
    \includegraphics[width=0.9\textwidth]{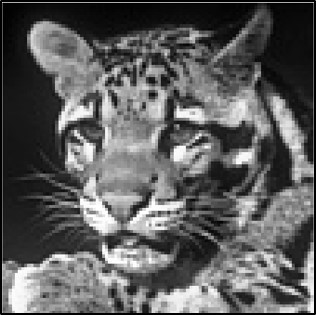}
    \caption{Initial\\ image}
\end{subfigure}
\begin{subfigure}{0.155\textwidth}
\centering
   \includegraphics[width=0.9\textwidth]{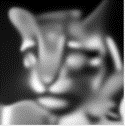}
  \caption{10 iterations}
\end{subfigure}
\begin{subfigure}{0.155\textwidth}
\centering
   \includegraphics[width=0.9\textwidth]{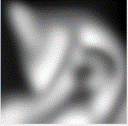}
  \caption{50\\ iterations}
\end{subfigure}
\begin{subfigure}{0.155\textwidth}
\centering
   \includegraphics[width=0.9\textwidth]{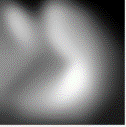}
    \caption{500 iterations}
\end{subfigure}
\begin{subfigure}{0.155\textwidth}
\centering
   \includegraphics[width=0.9\textwidth]{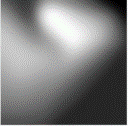}
   \caption{1300 iterations}
\end{subfigure}
\begin{subfigure}{0.155\textwidth}
\centering
    \includegraphics[width=0.9\textwidth]{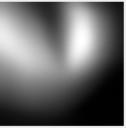}
    \caption{\\Stable mode}
\end{subfigure}
\begin{subfigure}{0.18\textwidth}
\centering
    \includegraphics[width=0.7\textwidth]{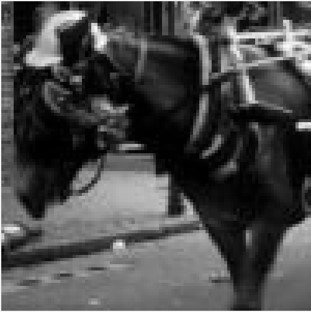}
    \caption{Initial\\ image}
\end{subfigure}
\begin{subfigure}{0.155\textwidth}
\centering
    \includegraphics[width=0.9\textwidth]{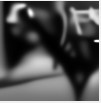}
    \caption{10 iterations}
\end{subfigure}
\begin{subfigure}{0.155\textwidth}
\centering
    \includegraphics[width=0.9\textwidth]{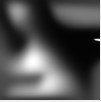}
    \caption{100 iterations}
\end{subfigure}
\begin{subfigure}{0.155\textwidth}
\centering
    \includegraphics[width=0.9\textwidth]{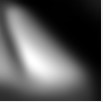}
    \caption{350 iterations}
\end{subfigure}
\begin{subfigure}{0.155\textwidth}
\centering
    \includegraphics[width=0.9\textwidth]{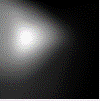}
    \caption{700 iterations}
\end{subfigure}
\begin{subfigure}{0.155\textwidth}
\centering
    \includegraphics[width=0.9\textwidth]{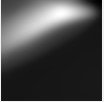}
    \caption{\\Stable mode}
\end{subfigure}
\vspace{-15pt}
\caption{FFDnet power method evolution to final stable modes (2000 iterations), validated following \Cref{Sec::approx_eig}.
Rows~$1$-$2$: initial condition \texttt{cameraman}, with $\lambda=1.0009$. 
Row~$3$: initial condition \texttt{tiger}, with $\lambda=1.0036$, $\theta_{2000}=0.07^\circ$. 
Row~$4$: initial condition \texttt{horse}, with $\lambda=1.0039$, $\theta_{2000}=0^\circ$.
}
\label{Fig::FFDnet}
\vspace{-25pt}
\end{figure*}

\paragraph{Further Stable Modes} \cref{Fig::FFDnet_horse_2nd_3rd} shows further stable modes using FFDnet with initial condition \texttt{horse}. We here follow the concept introduced in Algorithm~4 in~\cite{HAITFRAENKEL2021103041}. Given previously attained stable modes, we first perform a small number of power iterations with projections onto the space orthogonal to these modes, to ensure sufficient deviation and avoid recurrence of stable modes. We then perform the regular nonlinear power method to reach as state which meets~\cref{Eq:new_eigenproblem}.
Eigenvalues are now smaller than but very close to 1, as expected from a denoiser (see Cor.~1 in~\cite{HAITFRAENKEL2021103041}). In this example, the eigenvalues meet $\lambda_1>\lambda_2>\lambda_3$, but this cannot be guaranteed in general.

\begin{figure*}
\captionsetup[subfigure]{justification=centering}
\centering
\begin{subfigure}{0.155\textwidth}
\centering
    \includegraphics[height=2cm]{horse.jpg}
    \caption{Initial image}
\end{subfigure}
\begin{subfigure}{0.155\textwidth}
\centering
    \includegraphics[height=2cm]{FFDnet_horse_2000it.jpg}
    \caption{Stable mode $v_1$}
\end{subfigure}
\begin{subfigure}{0.155\textwidth}
\centering
    \includegraphics[height=2cm]{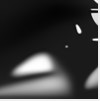}
    \caption{New initial condition for $v_2$}
\end{subfigure}
\begin{subfigure}{0.155\textwidth}
\centering
    \includegraphics[height=2cm]{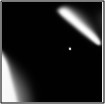}
    \caption{Stable mode $v_2$}
\end{subfigure}
\begin{subfigure}{0.155\textwidth}
\centering
    \includegraphics[height=2cm]{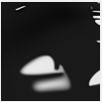}
    \caption{New initial condition for $v_3$}
\end{subfigure}
\begin{subfigure}{0.155\textwidth}
\centering
    \includegraphics[height=2cm]{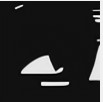}
    \caption{Stable mode $v_3$}
\end{subfigure}
\vspace{-15pt}
\caption{FFDnet further stable modes (initial condition \texttt{horse}), validated following \Cref{Sec::approx_eig}. For the stable modes $v_i$, fifty projection iterations enforce orthogonality to modes $v_{j}$ for $j<i$, generating a new initial condition. 2000 iterations of the power method then follow. 
(d): $v_2$, with $\lambda_2=0.9934$, $\theta_{2000}=0.26^\circ$. 
(f): $v_3$, with $\lambda_3=0.9815$, $\theta_{2000}=0.09^\circ$.  
}
\label{Fig::FFDnet_horse_2nd_3rd}
\end{figure*}

\paragraph{Unstable Modes} We follow Section~2.6 in~\cite{HAITFRAENKEL2021103041} to reveal unstable modes of the denoiser. We construct the complementary texture generator operator, $T^\dagger(u) := u-T(u)$.  \cref{Fig::FFDnet_smallest} shows stable modes of $T^\dagger$, with eigenvalues very close to but larger that 1. 
Modes were validated following \Cref{Sec::approx_eig}.
Note that stable modes of $T^\dagger$ are unstable modes of the denoising net $T$, characterized by fine textures. 

\begin{figure*}
\captionsetup[subfigure]{justification=centering}
\centering
\begin{subfigure}{0.2\textwidth}
\centering
    \includegraphics[height=2.2cm]{horse.jpg}
    \caption{\\Initial image}
\end{subfigure}
\begin{subfigure}{0.28\textwidth}
\centering
    \includegraphics[height=2.2cm]{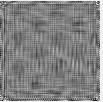}
    \caption{Unstable mode of denoiser}
\end{subfigure}
\begin{subfigure}{0.2\textwidth}
\centering
    \includegraphics[height=2.2cm]{tiger.jpg}
    \caption{\\Initial image}
\end{subfigure}
\begin{subfigure}{0.28\textwidth}
\centering
    \includegraphics[height=2.2cm]{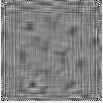}
    \caption{Unstable mode of denoiser}
\end{subfigure}
\vspace{-15pt}
\caption{Stable modes of the FFDnet texture generator using 10000 iterations, validated following \Cref{Sec::approx_eig}. These are \emph{unstable} modes of the FFDnet denoiser. 
(a)-(b): initial condition \texttt{horse}, with $\lambda=1.006$, $\theta_{10000}=0.23^\circ$.
(c)-(d): initial condition \texttt{tiger}, with $\lambda=1.0033$, $\theta_{10000}=0.4829^\circ$.
}
\label{Fig::FFDnet_smallest}
\vspace{-25pt}
\end{figure*}
\begin{figure*}
\captionsetup[subfigure]{justification=centering}
\begin{subfigure}{0.155\textwidth}
    \centering
    \includegraphics[height=2cm]{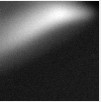}
    \caption{Stable mode with Gaussian noise}
\end{subfigure}
\begin{subfigure}{0.155\textwidth}
    \centering
    \includegraphics[height=2cm]{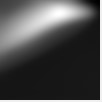}
     \caption{FFDnet Denoising of (a)}
\end{subfigure}
\begin{subfigure}{0.155\textwidth}
    \centering
    \includegraphics[height=2cm]{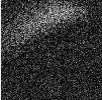}
    \caption{Noise removed from (a)}
\end{subfigure}
\begin{subfigure}{0.155\textwidth}
    \centering
    \includegraphics[height=2cm]{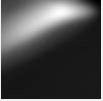}
    \caption{JPEG compression of stable mode}
\end{subfigure}
\begin{subfigure}{0.155\textwidth}
    \centering
    \includegraphics[height=2cm]{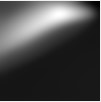}
     \caption{FFDnet Denoising of (d)}
\end{subfigure}
\begin{subfigure}{0.155\textwidth}
    \centering
    \includegraphics[height=2cm]{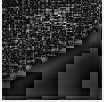}
    \caption{Artifacts removed from (d)}
\end{subfigure}
\\
\begin{subfigure}{0.24\textwidth}
\centering
\includegraphics[height=2cm]{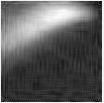}
\caption{Stable mode + unstable mode}
\end{subfigure}
\begin{subfigure}{0.2\textwidth}
\centering
\includegraphics[height=2cm]{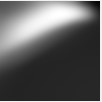}
\caption{Denoising \\of (g)}
\end{subfigure}
\begin{subfigure}{0.24\textwidth}
\centering
\includegraphics[height=2cm]{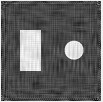}
\caption{Unstable mode + structure}
\end{subfigure}
\begin{subfigure}{0.28\textwidth}
\centering
\includegraphics[height=2cm]{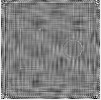}
\caption{Texture generator applied to (i)}
\end{subfigure}
\vspace{-15pt}
\caption{Robustness of FFDnet's stable mode to small degradations. (a)-(c): noise degradation with variance $0.01^2$ corrected by 10 power iterations. (d)-(f): JPEG compression degradation, corrected by 10 power iterations. (g)-(h): The unstable mode is considered as ``noise'' by the denoiser . (i)-(j): A structure is considered as ``noise'' by the texture generator.\label{Fig::degradation}}
\vspace{-20pt}
\end{figure*}
\paragraph{Robustness to Small Degradations} \cref{Fig::degradation} illustrates the robustness of stable modes to small degradations. A degraded stable mode has a similar Rayleigh quotient to that of the original stable mode. When applying the power method, it evolves back to the critical point in its neighborhood, given by the original stable mode. We also show how noise robustness holds in a very wide sense. The denoiser considers textures and fine details, such as the small eigenvector, as noise to be removed. On the other hand, the texture generator prefers noise and textures and it considers coarse structures as noise to be removed.

\subsubsection{DnCNN Deep Denoising Net}\label{Sec::DnCNN}
We present results of \cref{alg:nonhom_power_it} for the DnCNN deep denoising net~\cite{zhang2017beyond}.
\begin{figure*}[h!]
\captionsetup[subfigure]{justification=centering}
\centering
\begin{subfigure}{0.2\textwidth}
\centering
    \includegraphics[height=2.3cm]{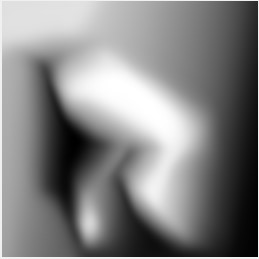}
    \caption{FFDnet's smoothing nature: iterative denoising}
    \label{Fig::FFDnet_nature}
\end{subfigure}
\begin{subfigure}{0.23\textwidth}
\centering
    \includegraphics[height=2.3cm]{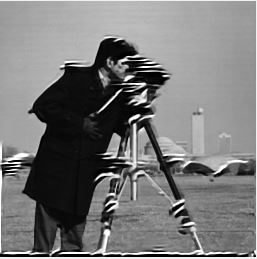}
    \caption{DnCNN's stripe sharpening effect: iterative denoising}
    \label{Fig::stripes_nature}
\end{subfigure}
\begin{subfigure}{0.16\textwidth}
\centering
    \includegraphics[height=2.3cm]{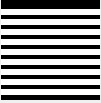}
    \caption{\\Horizontal stripes}
    \label{Fig::hor_str}
\end{subfigure}
\begin{subfigure}{0.16\textwidth}
\centering
    \includegraphics[height=2.3cm]{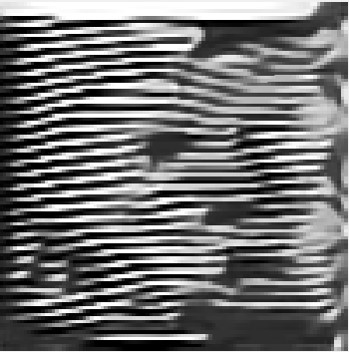}
    \caption{DnCNN's stable mode of (c)}
\end{subfigure}
\begin{subfigure}{0.18\textwidth}
\centering
    \includegraphics[height=2.3cm]{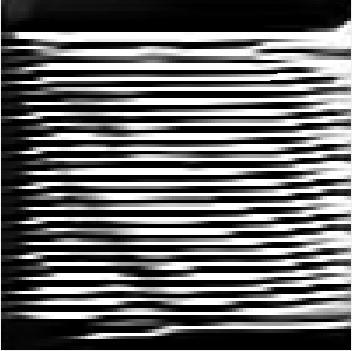}
    \caption{DnCNN's iterative denoising of (c)}
\end{subfigure}
\begin{subfigure}{0.15\textwidth}
\centering
    \includegraphics[height=2.3cm]{hor_strip.jpg}
    \caption{\\Horizontal stripes}
\end{subfigure}
\begin{subfigure}{0.16\textwidth}
\centering
    \includegraphics[height=2.3cm]{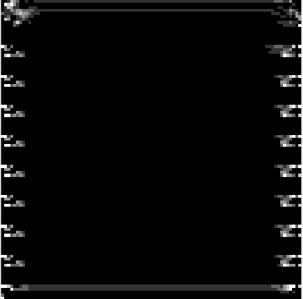}
    \caption{Difference when denoising (f), MSE=2.2}
\end{subfigure}
\begin{subfigure}{0.15\textwidth}
\centering
    \includegraphics[height=2.3cm]{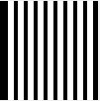}
    \caption{\\Vertical\\stripes}
\end{subfigure}
\begin{subfigure}{0.16\textwidth}
\centering
    \includegraphics[height=2.3cm]{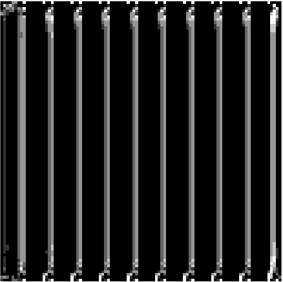}
    \caption{Difference when denoising (h), MSE=3.8}
\end{subfigure}
\begin{subfigure}{0.15\textwidth}
\centering
    \includegraphics[height=2.3cm]{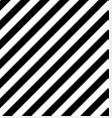}
    \caption{\\Diagonal stripes}
\end{subfigure}
\begin{subfigure}{0.16\textwidth}
\centering
    \includegraphics[height=2.3cm]{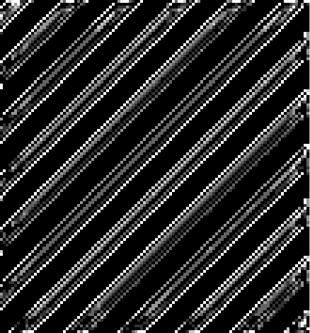}
    \caption{Difference when denoising (j), MSE=13.5}
\end{subfigure}
\begin{subfigure}{0.15\textwidth}
\centering
    \includegraphics[height=2.3cm]{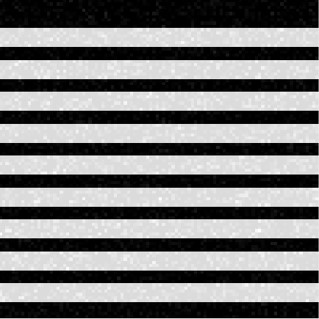}
    \caption{Noisy horizontal stripes}
\end{subfigure}
\begin{subfigure}{0.16\textwidth}
\centering
    \includegraphics[height=2.3cm]{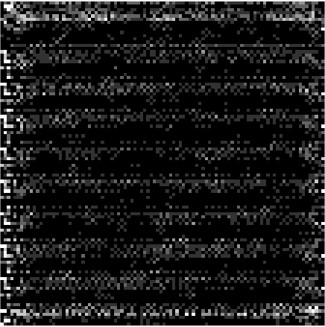}
    \caption{Difference when denoising (l), MSE=11.0}
\end{subfigure}
\begin{subfigure}{0.15\textwidth}
\centering
    \includegraphics[height=2.3cm]{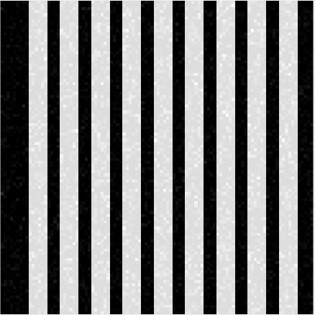}
    \caption{\\Noisy vertical stripes}
\end{subfigure}
\begin{subfigure}{0.16\textwidth}
\centering
    \includegraphics[height=2.3cm]{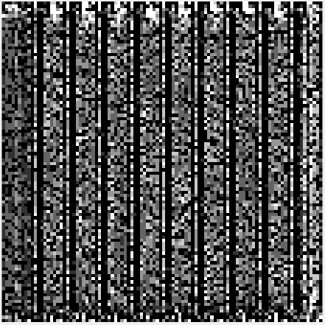}
    \caption{Difference when denoising (n), MSE=13.2}
\end{subfigure}
\begin{subfigure}{0.15\textwidth}
\centering
    \includegraphics[height=2.3cm]{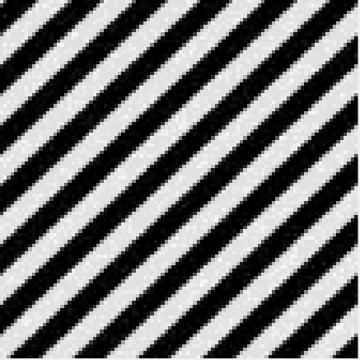}
    \caption{\\Noisy diagonal stripes}
\end{subfigure}
\begin{subfigure}{0.16\textwidth}
\centering
    \includegraphics[height=2.3cm]{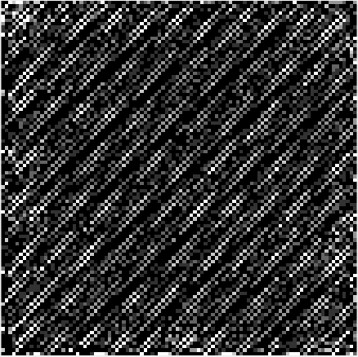}
    \caption{Difference when denoising (p), MSE=13.9}
\end{subfigure}
\caption{Nature of FFDnet and DnCNN. (a-b) 500 iterations of iterative denoising (no power method) of \texttt{cameraman}: FFDnet's smoothing nature (a) vs. DnCNN's stripe sharpening effect (b). (c)-(e): DnCNN "prefers" stripes: iterative denoising. \revise{(f)-(k): single denoising of clean stripe images of different orientations. Denoising almost does not affect horizontal stripes compared to others, since DnCNN "prefers" them (small MSE between input and output). (l-q): single denoising of stripe images of different orientations, with small noise. Better denoising is achieved when denoising horizontal stripes compared to others (small MSE between clean and denoised images). All difference images are enhanced with a multiplicative factor of~$50$.}}
\label{Fig::DnCNN_analysis}
\end{figure*}

\paragraph{Nature of Operator} Although it is a smoothing operator, DnCNN also produces a sharpening effect, adding stripes to the image. We first illustrate in \cref{Fig::stripes_nature} the horizontal stripes produced by iterative DnCNN denoising of the cameraman image, using 500 iterations. We also show the stable mode after 5000 iterations of the power method and iterative denoising result after 500 iterations for an synthetic stripe image. \revise{The second row of \cref{Fig::DnCNN_analysis} shows that even after a single application of DnCNN to horizontal stripes, in- and output are already almost collinear, with a small difference between them. Comparing this to stripe images of other orientations demonstrates this net's tendency to horizontal stripes. The third row of Fig. \ref{Fig::DnCNN_analysis} shows a single application of DnCNN to stripe images of different orientations, with small noise.
Also here horizontal stripes are better denoised compared to other orientations.}
DnCNN is a \emph{blind} denoiser that deals with unknown noise level, but assumes that noise is present in the input. Hence we believe the sharpening effect is caused by applying DnCNN iteratively throughout the power iterations to \emph{noiseless} and smooth images, for which it was not trained. 
\paragraph{Stable Modes} \cref{Fig::DnCNN} shows the power method evolution to the final stable mode, using DnCNN on two different initial images. As for FFDnet, eigenvalues are larger than but very close to 1.
The stable modes mainly are smoothed versions of the initial conditions with additional horizontal stripes. 

\begin{figure*}
\captionsetup[subfigure]{justification=centering}
\centering
\begin{subfigure}{0.155\textwidth}
\centering
    \includegraphics[width=0.9\textwidth]{cameraman.jpg}
    \caption{Initial image}
\end{subfigure}
\begin{subfigure}{0.155\textwidth}
\centering
    \includegraphics[width=0.9\textwidth]{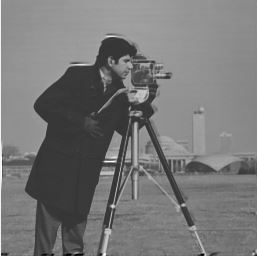}
   \caption{100 iterations}
\end{subfigure}
\begin{subfigure}{0.155\textwidth}
\centering
    \includegraphics[width=0.9\textwidth]{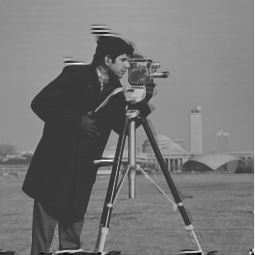}
    \caption{500 iterations}
\end{subfigure}
\begin{subfigure}{0.155\textwidth}
\centering
    \includegraphics[width=0.9\textwidth]{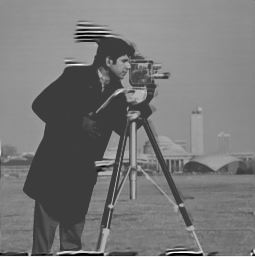}
    \caption{1000 iterations}
\end{subfigure}
\begin{subfigure}{0.155\textwidth}
\centering
    \includegraphics[width=0.9\textwidth]{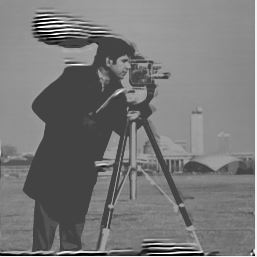}
    \caption{2000 iterations}
\end{subfigure}
\begin{subfigure}{0.155\textwidth}
\centering
    \includegraphics[width=0.9\textwidth]{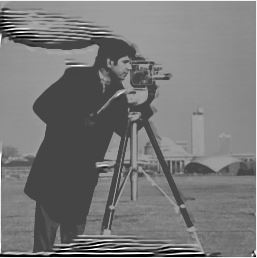}
    \caption{\\Stable mode}
\end{subfigure}
\\
\begin{subfigure}{0.18\textwidth}
\centering
    \includegraphics[width=0.7\textwidth]{horse.jpg}
    \caption{Initial\\ image}
\end{subfigure}
\begin{subfigure}{0.155\textwidth}
\centering
    \includegraphics[width=0.9\textwidth]{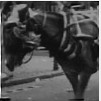}
    \caption{100 iterations}
\end{subfigure}
\begin{subfigure}{0.155\textwidth}
\centering
    \includegraphics[width=0.9\textwidth]{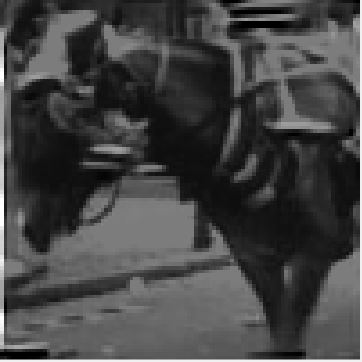}
    \caption{250 iterations}
\end{subfigure}
\begin{subfigure}{0.155\textwidth}
\centering
    \includegraphics[width=0.9\textwidth]{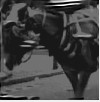}
    \caption{500 iterations}
\end{subfigure}
\begin{subfigure}{0.155\textwidth}
\centering
    \includegraphics[width=0.9\textwidth]{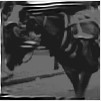}
    \caption{1000 iterations}
\end{subfigure}
\begin{subfigure}{0.155\textwidth}
\centering
    \includegraphics[width=0.9\textwidth]{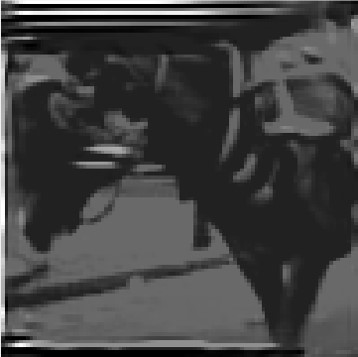}
    \caption{\\Stable mode}
\end{subfigure}
\vspace{-15pt}
\caption{DnCNN power method evolution to final stable modes after 5000 iterations, 
validated following \Cref{Sec::approx_eig}.
Row~$1$: initial condition \texttt{cameraman}, with $\lambda=1.0024$, $\theta_{2000}=0.13^\circ$. 
Row~$2$: initial condition \texttt{horse}, with $\lambda=1.0054$, $\theta_{2000}=0.24^\circ$. 
}
\label{Fig::DnCNN}
\vspace{-25pt}
\end{figure*}

\section{Conclusion and Outlook}
In this work we propose and analyze power methods to compute eigenvectors of proximal operators and neural networks. 
For proximal operators of one-homogeneous functionals we prove that a straightforward generalization of the linear power method converges to eigenvectors. 
Furthermore, our numerical experiments show the applicability of proximal eigenvectors, for instance in graph clustering and the computation of distance functions.
However, general denoising neural networks require a more general algorithm, which takes their natural domain into account. 
Despite the lack of theoretical convergence guarantees, this algorithm yields satisfactory approximate eigenvectors of the networks, which we interpret as (un)stable modes.
Those modes can be used to infer the kind of structures that the networks leave untouched or the type of noise they remove best, respectively.

Future work will include generalizations to proximal operators on Banach spaces, and the investigation of network architectures for which one can provably compute eigenvectors. 
Closely related to this is the design of 1-Lipschitz (non-expansive) networks, which are natural candidates for studying eigenvectors and fixed points.
A possible future application of (un)stable modes or eigenvectors of a network might consist in designing indicators for the amount of over-fitting, or for the generalization ability, respectively. 

\revision{%
Another challenging extension of our work deals with neural networks mapping between two different spaces, e.g., classification networks or networks for solving inverse problems.
In this case, one can study the doubly nonlinear eigenvalue problem $\lambda S(u) = T(u)$, where $S,T:U\to V$ are nonlinear operators. 
When $T$ is a neural network, it is not straightforward to choose $S$ and obtain a meaningful eigenvalue problem. 
One option consists in choosing $S:U\to V$ as a model-based operator which solves the same task as the neural network, for instance, a variational regularization method for solving an inverse problem. 
In this case, eigenvectors with large eigenvalues correspond to inputs on which the neural networks acts similar as the model-based approach.
Another example for $S$ could be the linearization of the network in the point $u$, leading to the problem $\lambda\langle \nabla T(u),u\rangle = T(u)$.
In any case, designing an appropriate power method for these problems is hard due to the lack of an operator mapping from $V$ to $U$. 
Here, one probably has to rely on linearization and inversion of the Jacobian matrix of the network.

Finally, we would like to determine conditions on neural networks such that their power method can be rigorously analyzed, similar to the proximal power method.
A first step in this direction might be~\cite{moeller2019controlling} where neural networks were trained to provably output a descent direction of some external energy. 
In light of the last paragraph in \Cref{Sec::def}, this would not be a suitable class of networks for a power method. However, using similar techniques as in \cite{moeller2019controlling} it might be possible to ``invert'' the network and to train networks which provably imply an energy descent of the power method as in \cref{cor:decrease_J}.
}

\section*{Acknowledgments}
This work was supported by the European Union’s Horizon 2020 research and innovation programme under the Marie Skłodowska-Curie grant agreement No 777826.
GG acknowledges support by the Israel Science Foundation (Grant No. 534/19) and by the Ollendorff Minerva Center. 
\bibliographystyle{siamplain}
\bibliography{bibliography}
\end{document}